\documentclass[11pt,reqno]{amsart}
\usepackage[margin=1in]{geometry}
\usepackage{amsmath, amsthm, amssymb, enumerate}
\usepackage[colorlinks=true, pdfstartview=FitV, linkcolor=blue,citecolor=blue, urlcolor=blue]{hyperref}
\usepackage[abbrev,lite,nobysame]{amsrefs}
\usepackage{times}
\usepackage[usenames,dvipsnames]{color}
\usepackage{mathtools}

\usepackage{graphicx}

\mathtoolsset{showonlyrefs=true}

\newtheorem{proposition}{Proposition}[section]
\newtheorem{theorem}[proposition]{Theorem}

\newtheorem{lemma}[proposition]{Lemma}
\theoremstyle{definition}
\newtheorem{definition}[proposition]{Definition}
\newtheorem{remark}[proposition]{Remark}

\numberwithin{equation}{section}

\newcommand{\indFn}[1]{1 \! \! 1_{#1}}
\newcommand{\E}{\mathbb{E}}
\newcommand{\Prb}{\mathbb{P}}
\newcommand{\Ma}{\mathcal{P}}

\newcommand{\GG}{\mathcal{G}}
\newcommand{\RR}{\mathbb{R}}
\newcommand{\NN}{\mathbb{N}}

\newcommand{\eps}{\varepsilon}
\newcommand{\e}{{\rm e}}
\newcommand{\dd}{{\rm d}}

\newcommand{\En}{\mathcal{H}}
\newcommand{\EnD}{\mathcal{E}}
\newcommand{\EnR}{\mathcal{H}_{r}}

\newcommand{\XX}{\mathcal{X}}
\newcommand{\XXL}{\mathcal{X}_{L^2}}
\newcommand{\XXH}{\mathcal{X}_{H^1}}

\newcommand{\tu}{{\tilde u}}

\begin{document}
%\markboth{}
%{}
\title[Invariant measures for the stochastic 1D compressible NSE]{Invariant measures for the stochastic one-dimensional compressible Navier-Stokes equations}

\author[M. Coti Zelati, N. Glatt-Holtz, and K. Trivisa]{Michele Coti Zelati, Nathan Glatt-Holtz, and Konstantina Trivisa}

\address{Department of Mathematics, Imperial College London, London, SW7 2AZ, UK}
\email{m.coti-zelati@imperial.ac.uk}

\address{Department of Mathematics,Tulane University, New Orleans LA 70118, USA}
\email{negh@tulane.edu}

\address{Department of Mathematics, University of Maryland, College Park, MD 20742, USA}
\email{trivisa@math.umd.edu}

\subjclass[2000]{35R60, 37L40, 35Q35, 60H30}

\keywords{Navier-Stokes system, compressible fluid, stochastic perturbation, invariant measure}

\begin{abstract}
We investigate the long-time behavior of solutions to a stochastically forced one-dimensional Navier-Stokes
system, describing the motion of a compressible viscous fluid, in the case of linear pressure law. We prove
existence of an invariant measure for the Markov process generated by strong solutions. We overcome the difficulties
of working with non-Feller Markov semigroups on non-complete metric spaces by generalizing the classical Krylov-Bogoliubov
method, and by providing suitable polynomial and exponential moment bounds on the solution, together with pathwise estimates.
\end{abstract}

\maketitle

\setcounter{tocdepth}{3}
%\tableofcontents

\section{Introduction}
We consider the question of existence of an invariant measure for the stochastically forced Navier-Stokes equations of 
compressible fluid flows in one space dimension. Formulated on the unit interval $(0,1)$, the equations read
\begin{align}
&\rho_t + (\rho u)_x = 0 \label{eq:cNSE:1}\\
&\dd (\rho u) + \left( \rho u^2 + A^2\rho\right)_x \dd t = u_{xx} \dd t + \rho \sigma \dd W,
\label{eq:cNSE:2}
\end{align}
where  $\rho=\rho(t,x)$ is the fluid density and  $u=u(t,x)$ is the fluid velocity.
System \eqref{eq:cNSE:1}-\eqref{eq:cNSE:2} is appropriately written in dimensionless form, and is 
supplemented by the initial conditions
\begin{align}\label{eq:initial1}
\rho(0,x)=\rho_0(x), \qquad  u(0,x)=u_0(x),
\end{align}
where $\rho_0,u_0$ are assigned functions defined on
the interval $[0,1]$, with the (strictly positive) initial density having normalized mass
$$
\int_0^1 \rho_0(x) \dd x = 1.
$$
In view of mass conservation, we obviously have
\begin{align}\label{eq:mean1}
\int_0^1 \rho(t,x) \dd x = 1, \qquad \forall t\geq 0.
\end{align}
Moreover, we assume homogeneous  Dirichlet boundary conditions for the
velocity $u$, namely
\begin{align}\label{eq:boundary1}
u(t,0)=u(t,1)=0, \qquad  \forall t\geq 0.
\end{align}
Above, the driving noise is given by a collection of independent white noise processes, suitably colored in space (see \eqref{eq:noiseform} below), while $A>0$ is a dimensionless parameter, inversely proportional to the Mach number. When considering strong (in the deterministic sense) solutions
to \eqref{eq:cNSE:1}-\eqref{eq:boundary1}, the natural phase space for our system is
\begin{align}\label{eq:XXset}
\XX=\left\{(\rho,u)\in H^1(0,1)\times H^1_0(0,1): \int_0^1 \rho(x) \dd x = 1, \ \rho>0\right\}.
\end{align}
The main result of this article is the following.
%\alert{talk about the topology}. Then, the Markov semigroup  $\{\Ma_t\}_{t\geq 0}$ associated to \eqref{eq:cNSE:1}-\eqref{eq:boundary1}
%is Feller on $C_b(\XX)$
\begin{theorem}\label{thm:main}
For every fixed $A>0$, the Markov semigroup $\{\Ma_t\}_{t\geq 0}$ associated to \eqref{eq:cNSE:1}-\eqref{eq:boundary1}
possesses an invariant probability measure $\mu_A\in \mathfrak{P}(\XXL)$. Furthermore,
\begin{align}\label{eq:ciaoci}
\int_{\XX} \left[A^2\|(\log\rho)_x\|_{L^2}^2+\|u_x\|_{L^2}^2\right]\dd\mu_A(\rho,u) \leq \|\sigma\|^2_{L^\infty}.
\end{align}
Here, $\XXL$ denotes the set $\XX$ in \eqref{eq:XXset}, endowed with the $L^2\times L^2$ metric, and $\|\sigma\|_{L^\infty}$ is
given by \eqref{eq:noise1} below.
\end{theorem}
The existence of statistically stationary states to randomly driven systems in fluid dynamics is of basic importance from 
both mathematical and experimental viewpoints. On the one hand, the existence of an invariant measure provides information
on the generic long-time behavior of the system. On the other hand, under ergodicity assumptions, it provides a link between
experimental observations (for example, in turbulence theory) and theoretical predictions.

In the context of fluid dynamics, a satisfactory theory of invariant measures has been developed for two-dimensional, incompressible
flows \cites{EMS01,HM06, Mattingly02,BKL02,BKL01,KS00,FM95,MY02,Hairer02,DPZ96,KS12}. In three dimensions, we mention the works 
\cites{DPD03,DPD08,FR08}, and especially \cite{FG95}, in which solutions which are strictly stationary stochastic processes are 
constructed, but the concept of invariant measure as a steady state is not well defined due to the absence of the Markov property. 
In this sense, the situation for multi-dimensional compressible flows is similar. To the best of our knowledge, the only relevant article 
 is the recent \cite{BFHM17}, in which statistically stationary solutions (but no invariant measure) are constructed. There are, however,
 results for both the dissipative and non-dissipative one-dimensional Burgers equation  \cites{Sinai91,EKMS00,GM05,DPD98,DPG95}.

In this article, we prove existence of an invariant measure for the compressible Navier-Stokes system. 
The main advantage in working in one space dimension is that the equations are globally well-posed and can be solved pathwise,  leaning on various  works in the deterministic setting \cites{KazhikhovShelukhin77,Solonnikov76, MV07,Hoff87,Hoff98,HS01,HZ00, HZ03}.  
 For the global solvability in the stochastic case, we mention \cites{TH96,TH98}, valid in one space dimension, and 
\cites{Tornatore00, Smith15, WW15, BH14, BFH16} for the multi-dimensional case. For the deterministic set-up, we refer
to the classical references \cites{Feireisl04,Lions98}, and to \cites{FP01, FP07} for long-time behavior results.

In turn, \eqref{eq:cNSE:1}-\eqref{eq:boundary1} generates a proper Markov semigroup, and the concept of an invariant measure 
can be defined in a standard manner. However, there are several obstructions on the path towards 
the proof of Theorem \ref{thm:main}.  The main ingredients of our approach and contributions to the existing theory  
on invariant measures can be summarized as follows:

\medskip 

\noindent $\bullet$
 {\em  An $L^2$-based continuity result via the derivation of polynomial and exponential moment bounds.} As it is is clear from its statement, the topology on $\XX$ plays an essential role. The
metric space  $\XXL$ is not complete (hence not Polish), due to both the open condition $\rho >0$ and the fact that the Sobolev 
space $H^1$ is not a closed subspace of $L^2$. This choice is dictated by the available continuous dependence 
estimates (cf. Theorem \ref{eq:contdiprel}), which make use of the so-called relative entropy functional and essentially provide an $L^2$-based
continuity result. Such a result crucially depends on polynomial and exponential moment bounds of the solution to 
\eqref{eq:cNSE:1}-\eqref{eq:boundary1}, which are carried out in detail in Section \ref{sec:eneraefd}, and on careful pathwise estimates (see
Section \ref{sub:pathest}).

\medskip 

\noindent $\bullet$
{\em Introduction of a (larger) class of functions on $\XXL$ 
which is invariant under the Markov semigroup $\Ma_t$.}
Closely related to the issue mentioned above, a second difficulty lies in the fact that the Markov semigroup 
associated to \eqref{eq:cNSE:1}-\eqref{eq:boundary1} 
is not known to be Feller (cf. Section \ref{sub:Markov}), namely $\Ma_t$ may not map  $C_b(\XXL)$ (the real-valued, continuous bounded functions on $\XXL$) into itself. This is again attributable to the mismatch between the set $\XX$ and the $L^2$ topology, and 
causes problems when trying to apply the classical Krylov-Bogoliubov 
procedure to prove existence and, in particular, invariance of a probability measure constructed as a subsequential limit of time-averaged measures. We circumvent this issue by defining a  class of functions on $\XXL$ that is slightly larger that $C_b(\XXL)$, 
which is invariant under the Markov semigroup $\Ma_t$, and still well-behaved in the above limiting procedure.

\medskip 

\noindent $\bullet$
 {\em Derivation of lower and upper bounds of the density with the  correct time-averaged growth.}  Lastly, the lack of instantaneous smoothing in the density equation \eqref{eq:cNSE:1} constitutes an obstacle towards compactness estimates
for time-averaged measures. However, an energy structure already exploited in \cites{HZ00,MV07} allows to obtain a dissipation 
term involving $\|(\log\rho)_x\|_{L^2}$, thanks to the linear pressure law adopted here. Consequently, upper and lower bounds on the
density (which, in 1D, can be proven with more general pressure laws) have the correct time-averaged growth and provide suitable 
tightness estimates for time-averaged measures (cf. Section \ref{sub:tight}).

\medskip 

\noindent $\bullet$
The present article is, according to our knowledge, the first  that presents rigorous results on the existence of invariant measures for compressible flows.

\begin{remark}[The low Mach number limit]
One may be tempted to study the behavior of the measures $\mu_A$ as $A\to \infty$, in the spirit of compressible-incompressible 
limits \cites{DGLM99, LM98, Masmoudi07,FN07,BFH15}. From \eqref{eq:ciaoci}, it is clear that, as $A\to \infty$, the density component
 of the measures concentrates on sets such that $\log \rho =0$, namely, $\rho=1$. However, \eqref{eq:cNSE:1} implies that $u_x=0$,
 and due to the homogeneous boundary conditions \eqref{eq:boundary1}, we deduce that also $u=0$. This is inconsistent with the
 second equation, as long as $\sigma\neq 0$. By replacing $\sigma$ with $A^{-\eta}\sigma$, for any $\eta>0$, it is easily seen 
 from \eqref{eq:ciaoci} that $\mu_A\to \delta_{(\rho=1,u=0)}$ as $A\to \infty$, essentially describing rigid body motion.
\end{remark}

\subsection*{Extensions and further developments}
Theorem \ref{thm:main} is in fact true for pressure laws of the type
$$
p(\rho)\sim \rho, \quad \text{as } \rho\to 0,
$$
and any polynomial growth as $\rho\to \infty$. The case of pure power law  $p(\rho)=\rho^\gamma$, for $\gamma>1$, remains open.
Another important question is the uniqueness (and hence ergodicity) and the attracting properties  of the measure $\mu_A$. It is
worth to point out that the problem is highly degenerate, since the noise only acts on one component of the phase space. We conjecture
that uniqueness holds, since the unforced system evolves (at a very slow rate) towards the unique steady state $(\rho,u)=(1,0)$.

\subsection*{Outline of the article}
Section \ref{sec:sol} is dedicated to the properties of solutions to the compressible Navier-Stokes equations. 
Energy estimates and moment bounds are proven in Section \ref{sec:eneraefd}, pathwise estimates in Section \ref{sub:pathest}, 
while uniqueness and continuous
dependence on data are discussed in Sections \ref{sub:unique} and \ref{sub:conti}, respectively. In Section 
\ref{sec:invar} we prove Theorem \ref{thm:main}, setting up the Markovian framework and discussing existence
of invariant measure for non-Feller Markov processes on non-complete metric spaces.

\subsection*{Notation and conventions}
Throughout the paper, $c$ will denote a \emph{generic} positive constant \emph{independent} of $A$, 
whose value may change even in the same line of a certain equation. In the same way, $C_1, C_2, \ldots$ denote
specific large deterministic  constants.
We will denote by $\XXL$ the set $\XX$ endowed with $L^2 \times L^2$ metric. Notice that $\XXL$ is clearly
a metric space. When we want to indicate the set $\XX$ endowed with the $H^1\times H^1$ metric, we write $\XXH$.
By $\mathcal{B}(\XXL)$, we denote the family of Borel subsets of $\XXL$. 
 With the symbol $M_b(\XXL)$ (resp. 
$C_b(\XXL)$) we refer to the set of all real valued bounded measurable (resp. bounded continuous)
functions on $\XXL$. Finally, $\mathfrak{P}(\XXL)$ is the set of all probability measures on $\XXL$.

\section{Existence and uniqueness of solutions}\label{sec:sol}
In view of the particular form of the stochastic forcing in \eqref{eq:cNSE:2}, the proof of 
existence and uniqueness of pathwise solutions borrows many ideas from the deterministic
case. Fix a stochastic basis 
$$
\mathcal{S}=(\Omega, \mathcal{F}, \Prb, \{\mathcal{F}_t\}_{t\geq 0}, W).
$$
We write the noise term in \eqref{eq:cNSE:2} as
\begin{align}\label{eq:noiseform}
 \sigma \dd W= \sum_{k=1}^\infty \sigma_k(x)\dd W^k(t).
\end{align}
The sequence $W(t) =\{W^k(t)\}_{k\in\NN}$ consists of independent copies of the standard one-dimensional
Wiener process 
(Brownian motion).  As such, for each $k$, $\dd W^k(t)$ is formally a 
white noise which, in particular, is stationary in time.
Throughout the article, we will assume 
\begin{align}\label{eq:noise}
\|\sigma_{xx}\|_{L^2}^2 :=\int_0^1 \sum_{\ell=1}^\infty |(\sigma_\ell)_{xx}(x)|^2\dd x<\infty, \qquad  \sigma_\ell\in H^2\cap H^1_0.
\end{align}
In particular, this implies 
\begin{align}\label{eq:noise1}
\|\sigma\|_{L^\infty}^2 :=\sup_{x\in (0,1)} \sum_{\ell=1}^\infty |\sigma_\ell(x)|^2<\infty.
\end{align}
By considering the change of variable $\tu=u-\sigma W$, a pathwise approach for the existence of
solutions (even in the multidimensional case) has been carried out in \cite{FMN13} (see also \cites{Tornatore00,TH96,TH98} and the 
more recent \cite{BFH16} for local existence of strong solutions). 
In the deterministic setting and in one space dimension, the local existence of strong solutions has been proved in \cite{Solonnikov76}, 
while their global existence can be found in \cite{MV07}, in the
case of pressure law of the type $\rho^\gamma$, for $\gamma>1$. That this result can be extended to our case is a consequence
of the estimates provided in the work \cite{HZ00} on (deterministic) weak solutions and their regularization properties. We summarize
these observations in the theorem below.

\begin{theorem}\label{T2.1}
Fix a stochastic basis $\mathcal{S}=(\Omega, \mathcal{F}, \Prb, \{\mathcal{F}_t\}_{t\geq 0}, W)$. For any $(\rho_0,u_0)\in \XX$,
there exists a unique $(\rho(t;\rho_0), u(t;u_0))$ satisfying \eqref{eq:cNSE:1}-\eqref{eq:boundary1}, in the time integrated sense 
and with the regularity
$$
\rho\in L_{loc}^\infty(0,\infty;H^1), \qquad u\in L_{loc}^\infty(0,\infty;H_0^1)\cap L_{loc}^2(0,\infty;H^2),
$$
almost surely. Moreover, for every $t\geq 0$,
$$
\rho(t)\in L^\infty \qquad \text{and}\qquad \frac{1}{\rho(t)}\in L^\infty,
$$
almost surely.
\end{theorem}
The bounds in the Theorem \ref{T2.1}  (see also Lemma \ref{lem:enbounds}) for the 
density show that neither vacuum states nor concentration states can
occur, no matter how large the initial datum is.
This is one of several important
differences between the Navier-Stokes equations and the inviscid Euler
equations, for which vacuum states may in fact occur for large initial data
and for certain equations of state (cf. \cites{CH,CG}). It is also relevant in
this regard that solutions of the Navier-Stokes equations show certain
instabilities when vacuum states are allowed (cf. Hoff and Serre \cite{HS}).

As far as the regularity of solutions is concerned, we need quantitative estimates and moment bounds that are new to 
the best of our knowledge. These estimates are contained in the the Propositions \ref{prop:allestimates} and \ref{prop:allestimates2} below,
and, in particular, agree with the regularity expressed in the theorem above.

\subsection{Energy estimates}\label{sec:eneraefd}
For a pair $(\rho,u)\in \XX$, we define the entropy function in the classical way as
\begin{align}
\En(\rho,u) = \int_0^1 \left(\frac{1}{2} \rho u^2 + A^2\rho \log\rho \right)\dd x.
\end{align}
In the one-dimensional setting, it turns out the strong solutions are also well-behaved with respect to 
the modified energy functional
\begin{align}
\EnD(\rho, u) 
= \En(\rho, u) + \frac{1}{2}  \int_0^1 \left(   \frac{\rho_x u}{\rho} +   \frac{1}{2} \frac{\rho_x^2}{\rho^3} \right) \dd x.
%&= \frac12 \En(\rho, u)+\frac14\int_0^1 \rho\left( u +\frac{\rho_x}{\rho^2} \right)^2 \dd x+  \frac{ A^2}{2}\int_0^1 \rho\log\rho \dd x .
\end{align}
Several quantities are controlled by $\EnD$. In particular, uniform estimates on $\EnD$ entail lower bounds on $\rho$. 
\begin{lemma}\label{lem:enbounds}
Assume that $(\rho,u)\in \XX$ are such that $\EnD(\rho,u)<\infty$. Then
\begin{align}\label{eq:unifbound1}
\e^{-\left[ 8\EnD(\rho,u) \right]^{1/2}} \leq \rho(x) \leq  \e^{\left[ 8\EnD(\rho,u) \right]^{1/2}}, \qquad \forall x\in [0,1],
\end{align}
and  
\begin{align}\label{eq:unifbound2}
\|\rho_x\|_{L^2}^2 \leq 8\EnD(\rho,u)\e^{3\left[ 8\EnD(\rho,u) \right]^{1/2}}  .
\end{align}
Moreover,
\begin{align}\label{eq:unifbound3}
\| u\|^2_{L^2} \leq 2\En(\rho,u ) \e^{\left[ 8\EnD(\rho,u) \right]^{1/2}}.
\end{align}
\end{lemma}

\begin{proof}
As a preliminary observation we note that (pointwise)
\begin{align}\label{eq:insfdgk}
\frac{\rho_x u}{\rho}\leq \rho u^2+\frac{1}{4}  \frac{\rho_x^2}{\rho^3}.
\end{align}
Moreover, since $\rho\in H^1$ and 
\begin{align}\label{eq:mn1}
\int_0^1\rho\, \dd x=1,
\end{align}
we have
\begin{align}
\int_0^1\rho \log\rho\dd x = \int_0^1\left[\rho \log \rho -\rho +1\right]\dd x\geq 0.
\end{align}
Therefore
\begin{align}\label{eq:inequalit1}
\EnD(\rho, u)&= \int_0^1 \left(\frac{1}{2} \rho u^2 +A^2\rho \log\rho+\frac{1}{2}  \frac{\rho_x u}{\rho} +   \frac{1}{4} \frac{\rho_x^2}{\rho^3} \right) \dd x\geq \frac18\int_0^1 \frac{\rho_x^2}{\rho^3}  \dd x.
\end{align}
By continuity of $\rho$ and \eqref{eq:mn1}, we can choose $x_0\in(0,1)$ such that $\rho(x_0)=1$ and write
$$
|\log \rho (x)|=\left|\int_{x_0}^x (\log\rho)_y\dd y\right|\leq\int_0^1 \frac{|\rho_y|}{\rho}\dd y.
$$
Hence,
$$
|\log \rho (x)|\leq\int_0^1 \frac{|\rho_x|\rho^{1/2}}{\rho^{3/2}}\dd x\leq \left[\int_0^1 \frac{\rho_x^2}{\rho^3} \dd x\right]^{1/2}.
$$
Consequently,
\begin{equation}
\|\log \rho\|_{L^\infty} \leq \left[\int_0^1 \frac{\rho_x^2}{\rho^3} \dd x\right]^{1/2}.
\end{equation}
From \eqref{eq:inequalit1}, the first bound \eqref{eq:unifbound1} follows immediately. Regarding \eqref{eq:unifbound2}, we combine
the upper bound in \eqref{eq:unifbound1} with  \eqref{eq:inequalit1} to deduce that
$$
\int_0^1 \rho_x^2\dd x\leq \| \rho\|_{L^\infty}^3\int_0^1 \frac{\rho_x^2}{\rho^3}\dd x\leq 8\EnD(\rho,u)\e^{3\left[ 8\EnD(\rho,u) \right]^{1/2}} 
$$
Finally, \eqref{eq:unifbound3} follows from the lower bound in \eqref{eq:unifbound1}.
\end{proof}

It is fairly clear from the above estimates that if $(\rho,u)\in \XX$, then $\EnD(\rho,u)<\infty$. Viceversa, if $(\rho,u)$ are smooth functions
such that  $\EnD(\rho,u)<\infty$, then $(\rho,u)\in \XX$. The above lemma provides a quantification of this dichotomy. 

It is crucial for us to prove various estimates on strong solutions to the compressible Navier-Stokes equations, keeping particular attention
to the dependence on the parameter $A>0$. We collect these estimates in the next two propositions, 
whose proofs are carried out in the subsequent
sections.

\begin{proposition}\label{prop:allestimates}
Consider a solution $(\rho,u)$ to \eqref{eq:cNSE:1}-\eqref{eq:boundary1} with initial data $(\rho_0,u_0)\in \XX$ in \eqref{eq:initial1} such that
$$
\EnD(\rho_0,u_0)<\infty.
$$
Then there hold
the entropy inequality
\begin{align}\label{eq:entro1}
   \E \En(\rho,u)(t) +  \E\int_0^t \|u_x \|_{L^2}^2 \dd s
   \leq  \En(\rho_0,u_0) +   \frac12\|\sigma\|^2_{L^\infty}t,
\end{align}
and the energy inequality
\begin{align}\label{eq:entro2}
 \E \EnD(\rho,u)(t) +\frac{1}{2}\E\int_0^t\|u_x\|^2_{L^2}\dd s + \frac{A^2}{2}\E\int_0^t\|(\log \rho)_x\|^2_{L^2} \dd s
\leq  \EnD(\rho_0,u_0) +  \frac{1}{2} \| \sigma \|^2_{L^\infty}t,
\end{align}
for all $t\geq 0$. Moreover, the exponential martingale estimate
\begin{align}\label{eq:entro3}
\Prb \left[\sup_{t\geq 0}\left(\EnD(\rho,u)(t) + \frac{1}{4}\int_0^t\left[\|u_x\|^2_{L^2} +A^2\|(\log \rho)_x\|^2_{L^2} \right]\dd s -\frac{1}{2}\|\sigma\|^2_{L^\infty}t\right)-\EnD(\rho_0,u_0)\geq R\right]
\leq \e^{-\gamma_0 R},\notag\\
\end{align}
holds for every $R\geq 0$ and for
\begin{align}
\gamma_0:= \frac{\min\left\{1,4A^2\right\}}{2\|\sigma\|^2_{L^\infty} }.
\end{align}
As a consequence, for every $m\geq 1$, there exists a constant $c_m>0$ such that
\begin{align}\label{eq:entro4}
&\E \sup_{t\in [0,T]} \left(\EnD(\rho,u)(t) + \frac{1}{4}\int_0^t\left[\|u_x\|^2_{L^2} +A^2\|(\log \rho)_x\|^2_{L^2} \right] \dd s \right)^m\notag\\
&\qquad\qquad\leq c_m\left(\EnD(\rho_0,u_0)^m+\|\sigma\|^{2m}_{L^\infty} T^m+\gamma_0^{-m}\right),
\end{align}
for every $T\geq 1$, and the exponential moment bound holds
\begin{align}\label{eq:entro5}
&\E\exp\left(\frac{\gamma_0}{2} \sup_{t\in [0,T]} \left(\EnD(\rho,u)(t) + \frac{1}{4}\int_0^t\left[\|u_x\|^2_{L^2} +A^2\|(\log \rho)_x\|^2_{L^2} \right]\dd s  \right)\right)\notag\\
&\qquad\qquad\leq\exp\frac{\gamma_0}{2}  \left(\EnD(\rho_0,u_0)+\frac{1}{2}\|\sigma\|^2_{L^\infty}T\right),
\end{align}q
for all $T\geq 1$.
\end{proposition}
The next result is instead a collection of similar, yet pathwise, estimates. 
\begin{proposition}\label{prop:allestimates2}
Consider a solution $(\rho,u)$ to \eqref{eq:cNSE:1}-\eqref{eq:boundary1} with initial data $(\rho_0,u_0)\in \XX$.
For every $T\geq 0$ and almost surely, there holds
\begin{align}\label{eq:energypath}
\sup_{t\in [0,T]}\left[\EnD(\rho(t),u(t))+\int_0^T\left[\|u_x\|_{L^2}^2 + A^2 \| (\log \rho)_x\|_{L^2}^2\right]\dd t\right]
\leq C_1(\sigma W,A, \EnD(\rho_0,u_0)).
\end{align}
In particular,
\begin{align}\label{eq:uplowpath}
\sup_{t\in [0,T]}\left[\|\rho(t)\|_{L^\infty}+\|\rho^{-1}(t)\|_{L^\infty}\right]\leq
C_2(\sigma W,A, \EnD(\rho_0,u_0)).
\end{align}
Moreover,
\begin{align}\label{eq:H1path}
\sup_{t\in [0,T]}\left[\|u_x(t)\|_{L^2}^2+\int_0^T\left[ \int_0^1\frac{|u_{xx}|^2}{\rho}\dd x\right]\dd t\right]\leq
C_3(\sigma W,A, \EnD(\rho_0,u_0),\|(u_0)_x\|_{L^2}).
\end{align}
In the above the deterministic constants $C_i$ can be explicitly computed.
\end{proposition}
The proof of \eqref{eq:uplowpath} is simply a consequence of \eqref{eq:unifbound1}.
We begin with Proposition \ref{prop:allestimates}.

\subsubsection{Entropy estimates}
According to the It\={o} lemma 
\begin{align}
\dd ( \rho u^2) &=  \dd \left(\frac{1}{\rho} (\rho u) ^2\right) 
  	= - \rho_t u^2  + \frac{1}{\rho}( 2 \rho u \cdot \dd (\rho u) + \dd (\rho u) \dd (\rho u) ) \notag\\
	&= (\rho u)_x u^2 \dd t +  2 u \left[  u_{xx} - \left( \rho u^2 +A^2\rho \right)_x \right] \dd t + 
	2 \sum_{\ell=1}^\infty \rho u \sigma_\ell \dd W^\ell + \sum_{\ell=1}^\infty\rho |\sigma_\ell|^2 \dd t. \notag
\end{align}
Hence, integrating by parts,
\begin{align}
\dd \frac12\int_0^1\rho u^2 \dd x + \| u_x\|_{L^2}^2\dd t =    - A^2\int_0^1\rho_x u\dd x \dd t + 
\sum_{\ell=1}^\infty\left(\int_0^1  \rho u \sigma_\ell \dd x \right) \dd W^\ell
+  \frac12\sum_{\ell=1}^\infty \int_0^1\rho |\sigma_\ell|^2  \dd x \dd t. \notag
\end{align}
Moreover, by integration by parts and using only \eqref{eq:cNSE:1}, we have
\begin{align*}
\frac{\dd}{\dd t}\int_0^1 \rho \log\rho\dd x 
= -\int_0^1 (\rho u)_x \left[\log\rho +1\right]\dd x= \int_0^1 \rho_x u \dd x.
\end{align*}
Combining the above computations we find
\begin{align}
   \dd \En (\rho,u)+ \| u_x\|_{L^2}^2\dd t =   \sum_{\ell=1}^\infty\left(\int_0^1  \rho u \sigma_\ell \dd x \right) \dd W^\ell +   \frac12\sum_{\ell=1}^\infty\int_0^1\rho |\sigma_\ell|^2  \dd x \dd t,
   \label{eq:entr:Bal}
\end{align}
so that, for every $t\geq 0$, we find the entropy balance
\begin{align}
   \E \En(\rho,u)(t) +  \E\int_0^t \|u_x \|^2 \dd s
   =   \En(\rho_0,u_0) +   \frac12\E\sum_{\ell=1}^\infty\int_0^t\int_0^1\rho |\sigma_\ell|^2  \dd x \dd s.
\end{align}
Notice that by \eqref{eq:mean1} and \eqref{eq:noise1}, we have
\begin{align}
 \sum_{\ell=1}^\infty\int_0^t\int_0^1\rho |\sigma_\ell|^2  \dd x \dd s\leq \|\sigma\|^2_{L^\infty}t,
\end{align}
so that 
\begin{align}
   \E \En(\rho,u)(t) + \E \int_0^t \|u_x \|^2 \dd s
   \leq   \En(\rho_0,u_0) +   \frac12\|\sigma\|^2_{L^\infty}t.
\end{align}
This is precisely \eqref{eq:entro1}.
\subsubsection{Energy estimates}\label{sub:ener}
Let $\ell = \rho^{-1/2}$ so that 
\begin{align*}
	&\ell_x = - \frac{1}{2} \frac{\rho_x}{\rho^{3/2}}, \qquad 
	\ell_x^2 = \frac{1}{4} \frac{\rho_x^2}{\rho^3},\\
	&(\ell_x)_t 
		= -\frac{1}{2} \frac{(\rho_x)_t}{\rho^{3/2}} + \frac{3}{4}  \frac{\rho_x  \rho_t}{\rho^{5/2}}
		= \frac{1}{2} \frac{(\rho u)_{xx}}{\rho^{3/2}} - \frac{3}{4}  \frac{\rho_x (\rho u)_x}{\rho^{5/2}},\\
	&\ell_{xx} = -\frac{1}{2} \frac{\rho_{xx}}{\rho^{3/2}} + \frac{3}{4}  \frac{\rho_x^2}{\rho^{5/2}}.
\end{align*} 
As such
\begin{align*}
(\ell_x^2)_t + (\ell_x^2 u)_x &= - \frac{\rho_x}{\rho^{3/2}} ( (\ell_x)_t + \ell_{xx} u) +  \ell_x^2 u_x
  = -\frac{\rho_x}{\rho^{3/2}} \left(  \frac{1}{2} \frac{2 \rho_x u_x + \rho u_{xx}}{\rho^{3/2}} - \frac{3}{4}  \frac{\rho_x \rho u_x}{\rho^{5/2}} \right) 
     +  \frac{1}{4} \frac{\rho_x^2  u_x}{\rho^3} \\
  &= - \frac{1}{2}  \frac{\rho_x u_{xx}}{\rho^{2}}.
\end{align*}
Notice that this computation only makes use of \eqref{eq:cNSE:1}.  On the other hand
\begin{align*}
   \dd \left( \frac{\rho_x u}{\rho}\right) 
   	&= \rho u  \left( \frac{\rho_x }{\rho^2} \right)_t + \frac{\rho_x }{\rho^2} \dd (\rho u) \\
   	&= \rho u \left( \frac{(\rho_x)_t }{\rho^2}  - 2 \frac{ \rho_x   \rho_t}{\rho^3} \right) \dd t 
	+  \frac{\rho_x }{\rho^2} \left( u_{xx} - \left( \rho u^2 + A^2\rho \right)_x \right)\dd t  + \sum_{\ell=1}^\infty\frac{\rho_x }{\rho}  \sigma_\ell \dd W^\ell\\
	&=  \frac{\rho_x u_{xx} }{\rho^2}  \dd t
	 -   \left( \frac{u(\rho u)_{xx} }{\rho}  - 2 \frac{ u \rho_x  (\rho u)_x }{\rho^2} +  \frac{\rho_x \left( \rho u^2 + A^2\rho \right)_x}{\rho^2} \right) \dd t 
		+ \sum_{\ell=1}^\infty\frac{\rho_x }{\rho}  \sigma_\ell \dd W^\ell\\
%	&= \nu \frac{\rho_x u_{xx} }{\rho^2}  dt - \left( (u u_x)_x - u_x^2  
%			+ \frac{2 u\rho_x u_x + u^2 \rho_{xx}}{\rho} -   \frac{ 2 u \rho_x   u_x }{\rho} 
%			- \frac{ 2 u^2 \rho_x^2 }{\rho^2}  + \frac{\rho_x^2 u^2  + 2\rho_x \rho u u_x}{\rho^2}  + \frac{ \Pl'(\rho) \rho_x^2}{\rho^2} \right) + \frac{\rho_x }{\rho}  \sigma dW\\
	&=   \frac{\rho_x u_{xx} }{\rho^2}  \dd t - \left( (u u_x)_x - u_x^2 
		+   \frac{ 2  \rho_x  u  u_x }{\rho} - \frac{\rho_x^2 u^2}{\rho^2} + \frac{u^2 \rho_{xx}}{\rho} 
			+ A^2\frac{ \rho_x^2}{\rho^2}\right) \dd t
		+ \sum_{\ell=1}^\infty\frac{\rho_x }{\rho}  \sigma_\ell \dd W^\ell\\
	&=   \frac{\rho_x u_{xx} }{\rho^2}  \dd t - \left( (u u_x)_x - u_x^2 
		+ \left( \frac{\rho_x u^2 }{\rho} \right)_x
			+ A^2\frac{\rho_x^2}{\rho^2}\right) \dd t
		+ \sum_{\ell=1}^\infty\frac{\rho_x }{\rho}  \sigma_\ell \dd W^\ell.
\end{align*}
Combining the above two equations and weighting appropriately we infer
\begin{align*}
  \dd \left(   \frac{\rho_x u}{\rho} + 2  \ell_x^2 \right) +  \left( \frac{\rho_x u^2 }{\rho} +u u_x+ 2  \ell_x^2 u \right)_x \dd t  + A^2\frac{\rho_x^2}{\rho^2} \dd t
 =   u_x^2 \dd t  + \sum_{\ell=1}^\infty\frac{\rho_x }{\rho}  \sigma_\ell \dd W^\ell,
\end{align*}
and hence
\begin{align}
  \dd \int_0^1\left(   \frac{\rho_x u}{\rho} + \frac{1}{2} \frac{\rho_x^2}{\rho^3} \right)\dd x  +A^2\int_0^1  \frac{\rho_x^2}{\rho^2}  \dd x \dd t
	 = \|u_x\|^2_{L^2} \dd t+ \sum_{\ell=1}^\infty\left(\int_0^1 \frac{\rho_x }{\rho}   \sigma_\ell \dd x \right) \dd W^\ell.
	 \label{eq:insane:bal}
\end{align}
We sum together \eqref{eq:entr:Bal} and half of \eqref{eq:insane:bal} to finally obtain
\begin{align}\label{eq:enrbval}
   \dd \EnD(\rho,u) + \frac{1}{2}\|u_x\|^2_{L^2}\dd t + \frac{A^2}{2}\int_0^1 \frac{\rho_x^2}{\rho^2} \dd x \dd t  
   	= &\sum_{\ell=1}^\infty\left( \int_0^1 \left(  \rho u + \frac{1}{2} \frac{\rho_x }{\rho} \right)\sigma_\ell \dd x\right) \dd W^\ell \notag\\
	&+ \frac{1}{2} \sum_{\ell=1}^\infty\int_0^1\rho |\sigma_\ell|^2  \dd x \dd t.
\end{align}
Hence
\begin{align}
 \E \EnD(\rho,u)(t) +\frac{1}{2}\E\int_0^t\|u_x\|^2_{L^2}\dd s + \frac{A^2}{2}\E\int_0^t\int_0^1 \frac{\rho_x^2}{\rho^2} \dd x \dd s
   	=   \EnD(\rho_0,u_0) +  \frac{1}{2} \E \sum_{\ell=1}^\infty\int_0^t \int_0^1\rho |\sigma_\ell|^2  \dd x \dd s,
\end{align}
so that
\begin{align}
 \E \EnD(\rho,u)(t) +\frac{1}{2}\E\int_0^t\|u_x\|^2_{L^2}\dd s + \frac{A^2}{2}\E\int_0^t\int_0^1 \frac{\rho_x^2}{\rho^2} \dd x \dd s
\leq  \EnD(\rho_0,u_0) +  \frac{1}{2} \| \sigma \|^2_{L^\infty}t,
\end{align}
and therefore \eqref{eq:entro2} holds. 

\subsubsection{Exponential estimates}\label{sub:expest}

It remains to prove \eqref{eq:entro3}, which follows from \eqref{eq:enrbval} and an application of the exponential martingale 
estimate
\begin{align}\label{eq:expmart}
\Prb \left[\sup_{t\geq 0}\left(Z(t) - \frac{\gamma}{2} \langle Z\rangle (t)\right)\geq R\right]\leq \e^{-\gamma R}, \qquad \forall R,\gamma>0,
\end{align}
valid for any continuous martingale $\{Z(t)\}_{t\geq 0}$ with quadratic variation $\langle Z\rangle (t)$. Indeed, consider 
the time integrated version of \eqref{eq:enrbval}, namely
\begin{align}\label{eq:enrbvalint}
  \EnD(\rho,u)(t) + \frac{1}{2}\int_0^t\|u_x\|^2_{L^2}\dd s + \frac{A^2}{2}\int_0^t\|(\log \rho)_x\|^2_{L^2} \dd s  
   	&=   \EnD(\rho_0,u_0)
	+Z(t)+ \frac{1}{2} \sum_{\ell=1}^\infty\int_0^t\int_0^1\rho |\sigma_\ell|^2  \dd x \dd s.
\end{align}
where $Z(t)$ is the martingale 
\begin{align}
Z(t)=\sum_{\ell=1}^\infty\int_0^t\left( \int_0^1 \left(  \rho u + \frac{1}{2} \frac{\rho_x }{\rho} \right)\sigma_\ell \dd x\right) \dd W^\ell
\end{align}
with quadratic variation
\begin{align}
\langle Z\rangle (t)=\sum_{\ell=1}^\infty\int_0^t \left(\left[\int_0^1 \rho u \sigma_\ell \dd x\right]^2 
+\frac{1}{4}\left[\int_0^1 \frac{\rho_x }{\rho} \sigma_\ell \dd x\right]^2 \right)\dd s.
\end{align}
In view of the Poincar\'e-like  inequality
\begin{align}\label{eq:poinc1}
\int_0^1 \rho u^2\dd x \leq \|u_x\|_{L^2}^2,
\end{align}
the mass constraint \eqref{eq:mean1} and  \eqref{eq:noise1}, we have
\begin{align}
\sum_{\ell=1}^\infty\int_0^t \left[\int_0^1 \rho u \sigma_\ell \dd x\right]^2 \dd s \leq \|\sigma\|_{L^\infty}^2\int_0^t \|u_x\|_{L^2}^2\dd s
\end{align}
and
\begin{align}
\frac{1}{4}\sum_{\ell=1}^\infty\int_0^t \left[\int_0^1 \frac{\rho_x }{\rho} \sigma_\ell \dd x\right]^2 \dd s\leq \frac{1}{4}\|\sigma\|_{L^\infty}^2\int_0^t \|(\log\rho)_x\|_{L^2}^2\dd s.
\end{align}
As a consequence, the quadratic variation of $Z(t)$ can be estimated as
\begin{align}\label{eq:quadvar}
\langle Z\rangle (t)\leq \|\sigma\|^2_{L^\infty} \int_0^t \left[\|u_x\|_{L^2}^2+\frac14\|(\log\rho)_x\|_{L^2}^2 \right]\dd s.
\end{align}
Now, from \eqref{eq:enrbvalint} we infer that the functional
\begin{align}
\Psi(t)=\EnD(\rho,u)(t) + \frac{1}{4}\int_0^t\|u_x\|^2_{L^2}\dd s + \frac{A^2}{4}\int_0^t\|(\log \rho)_x\|^2_{L^2} \dd s 
\end{align}
satisfies with probability one the inequality
\begin{align}\label{eq:psi1}
\Psi(t)-\frac{1}{2}\|\sigma\|^2_{L^\infty}t\leq &\EnD(\rho_0,u_0)+ \left[Z(t)- \frac{\gamma_0}{2} \langle Z\rangle (t)\right]\notag\\
&+ \frac{\gamma_0}{2} \langle Z\rangle (t)- \frac{1}{4}\int_0^t\|u_x\|^2_{L^2}\dd s - \frac{A^2}{4}\int_0^t\|(\log \rho)_x\|^2_{L^2} \dd s 
\end{align}
where we conveniently fix the constant $\gamma_0>0$ as
\begin{align}
\gamma_0:= \frac{\min\left\{1,4A^2\right\}}{2\|\sigma\|^2_{L^\infty} }.
\end{align}
In this way, from \eqref{eq:quadvar} it follows that
\begin{align}
\frac{\gamma_0}{2} \langle Z\rangle (t)- \frac{1}{4}\int_0^t\|u_x\|^2_{L^2}\dd s - \frac{A^2}{4}\int_0^t\|(\log \rho)_x\|^2_{L^2} \dd s \leq 0
\end{align}
and thus \eqref{eq:psi1} implies 
\begin{align}\label{eq:psi2}
\Psi(t)-\frac{1}{2}\|\sigma\|^2_{L^\infty}t\leq \EnD(\rho_0,u_0)+ \left[Z(t)- \frac{\gamma_0}{2} \langle Z\rangle (t)\right].
\end{align}
In turn, from \eqref{eq:expmart} and the above \eqref{eq:psi2} we deduce that
\begin{align}\label{eq:entro3equiv}
\Prb \left[\sup_{t\geq 0}\left(\Psi(t)-\frac{1}{2}\|\sigma\|^2_{L^\infty}t\right)-\EnD(\rho_0,u_0)\geq R\right]
\leq \Prb \left[\sup_{t\geq 0}\left(Z(t)- \frac{\gamma_0}{2} \langle Z\rangle (t)\right)\geq R\right]\leq \e^{-\gamma_0 R},
\end{align}
for every $R\geq0$. This is exactly \eqref{eq:entro3}.

\subsubsection{Polynomial and exponential moments}\label{sub:polyexp}
We now focus on \eqref{eq:entro4}, which is, in fact, an easy consequence of  \eqref{eq:entro3}. Indeed,  \eqref{eq:entro3equiv}
implies that
\begin{align}\label{eq:entro3equiv2}
\Prb\left[\sup_{t\in [0,T]}\Psi(t) - \EnD(\rho_0,u_0)-\frac{1}{2}\|\sigma\|^2_{L^\infty}T\geq R  \right]\leq \e^{-\gamma_0 R},
\end{align}
where $\Psi$ is the functional defined in \eqref{eq:psi1}. Using that, for a non-negative random variable $Z$, any $m\geq 1$ and any constant $c>0$, there holds
\begin{align}
\E[Z^m]\leq 2^{m-1}\E[(Z-c)^m \indFn{Z>c}]+ 2^{m-1}c^m
= 2^{m-1} \int_0^\infty  \Prb[Z-c> \lambda^{1/m}] \dd \lambda + 2^{m-1}c^m,
\end{align}
we infer from \eqref{eq:entro3equiv2} that for some constant $c_m>0$, independent of $\gamma_0, T$ and $\EnD(\rho_0,u_0)$ there holds
\begin{align}
\E\left[\sup_{t\in [0,T]}\Psi(t)^m\right]&\leq 2^{m-1}\int_0^\infty  \e^{-\gamma_0 \lambda^{1/m}}\dd \lambda + 2^{m-1}\left(\EnD(\rho_0,u_0)+\frac{1}{2}\|\sigma\|^2_{L^\infty}T\right)^m\notag\\
&\leq c_m\left(\gamma_0^{-m}+\EnD(\rho_0,u_0)^m+\|\sigma\|^{2m}_{L^\infty} T^m\right),
\end{align}
which is precisely \eqref{eq:entro4}. For \eqref{eq:entro5}, the idea is similar. Indeed, for a non-negative random variable $Z$ 
and any $\delta,c>0$, we also have that
that
\begin{align}
\E[\e^{\delta Z}]\leq \e^{\delta c}\E\left[\e^{\delta (Z-c)\indFn{Z>c}}\right]
= \e^{\delta c}\int_1^\infty  \Prb\left[Z-c> \frac{\ln\lambda}{\delta}\right] \dd \lambda.
\end{align}
Hence, for any $\delta<\gamma_0$, we deduce from \eqref{eq:entro3equiv2} that
\begin{align}
\E\left[\exp\left(\delta \sup_{t\in [0,T]}\Psi(t)\right)\right]&\leq\exp\left[\delta \left(\EnD(\rho_0,u_0)+\frac{1}{2}\|\sigma\|^2_{L^\infty}T\right)\right] \int_1^\infty \frac{1}{\lambda^{\gamma_0/\delta}}\dd \lambda\notag \\
&= \frac{\delta}{\gamma_0-\delta}\exp\left[\delta \left(\EnD(\rho_0,u_0)+\frac{1}{2}\|\sigma\|^2_{L^\infty}T\right)\right],
\end{align}
and  \eqref{eq:entro5} simply follows by choosing $\delta=\gamma_0/2$. We concluded the proof of Proposition \ref{prop:allestimates}.

\subsection{Pathwise estimates}\label{sub:pathest}
By setting $\tu=u-w$, with $w=\sigma W$ we obtain from \eqref{eq:cNSE:1}-\eqref{eq:cNSE:2} that the pair $(\rho,\tu)$
satisfies the system
\begin{align}
& \rho_t + (\rho \tu)_x = f(\rho,w) \label{eq:cNSE:1til}\\
& (\rho \tu)_t + ( \rho \tu^2 + A^2\rho )_x  =  \tu_{xx}+\tu  f(\rho,w)+\rho g(\rho,\tu,w),
\label{eq:cNSE:2til}
\end{align}
where
\begin{align}\label{eq:formfg}
f(\rho,w)=-(\rho w)_x, \qquad g(\rho,\tu,w)=\frac{w_{xx}}{\rho} -(\tu w)_x-ww_x,
\end{align}
with initial datum $(\rho_0,\tu_0)=(\rho_0,u_0)$.
Notice that from \eqref{eq:cNSE:2til} we can also write the velocity equation alone as
\begin{align}
\tu_t +  \tu\tu_x + A^2\frac{\rho_x}{\rho}   =  \frac{\tu_{xx}}{\rho}+ g(\rho,\tu,w),
\label{eq:cNSE:3til}
\end{align}
which will also be useful in the sequel. Without explicit reference, we will use \eqref{eq:noise1} multiple times
to bound the various norms of $w$ appearing in the estimates.

\subsubsection{Pathwise energy estimates}
We begin by proving \eqref{eq:energypath}.
A lengthy calculation analogous to that in Section \ref{sub:ener} shows that
\begin{align}
\frac{\dd}{\dd t} \EnD(\rho,\tu) + \frac{1}{2} \|\tu_x\|_{L^2}^2 + \frac{A^2}{2} \| (\log \rho)_x\|_{L^2}^2= \mathfrak{F} (\rho,\tu, f,g),
    \label{eq:E:evo:gf}
\end{align}
where
\begin{align}
\mathfrak{F}(\rho,\tu, f,g)=& \int_0^1 \left( \frac{1}{2} f \tu^2   + A^2 f (\log \rho + 1) +
\frac{1}{2} \frac{\rho_x }{\rho^2}f \tu  + 
              \frac{1}{2}\frac{f_x \tu}{\rho} -  \frac{\rho_x \tu f }{\rho^2}
               + \frac{1}{2}\frac{\rho_x f_x }{\rho^3} - \frac{3}{4}\frac{(\rho_x)^2 f}{\rho^4}
        \right) \dd x\notag \\
    &+ \int_0^1 \left( \rho \tu g + \frac{1}{2} \frac{\rho_x }{\rho} g \right) \dd x.
\end{align}
Recalling \eqref{eq:formfg}, we may rewrite the above term as
\begin{align}
\mathfrak{F}(\rho,\tu, f,g) 
= -A^2\int_0^1(\rho w)_x \log \rho\,\dd x
-\frac{1}{2}\int_0^1 \left[\frac{\rho_x (\rho w)_{xx} }{\rho^3} - \frac{3}{2}\frac{(\rho_x)^2 (\rho w)_x}{\rho^4}
-  \frac{\rho_x }{\rho} \left(\frac{w_{xx}}{\rho}-ww_x\right)\right] \dd x \notag\\
-\frac{1}{2}\int_0^1 \left[  \left( \tu^2   -   \frac{\rho_x \tu}{\rho^2}\right)(\rho w)_x
+\frac{(\rho w)_{xx} \tu}{\rho} -2\rho \tu \left(\frac{w_{xx}}{\rho} -(\tu w)_x-ww_x\right) +  \frac{\rho_x }{\rho}  (\tu w)_x \right] \dd x.
\end{align}
We now integrate by parts multiple times. Clearly,
\begin{align}
-A^2\int_0^1(\rho w)_x \log \rho\,\dd x=A^2\int_0^1 w\rho_x\,\dd x=-A^2\int_0^1 w_x\rho\,\dd x.
\end{align}
Moreover, expanding all the derivatives and integrating by parts, the second term above reduces to
\begin{align}
&\int_0^1 \left[\frac{\rho_x (\rho w)_{xx} }{\rho^3} - \frac{3}{2}\frac{(\rho_x)^2 (\rho w)_x}{\rho^4}-  
\frac{\rho_x }{\rho} \left(\frac{w_{xx}}{\rho}-ww_x\right)\right] \dd x\notag\\
&\qquad\qquad =\int_0^1 \left[\frac12\frac{\left[(\rho_x)^2\right]_x w }{\rho^3}+\frac12\frac{(\rho_x)^2w_x }{\rho^3} - \frac{3}{2}\frac{(\rho_x)^3  w}{\rho^4}+\frac{\rho_x }{\rho} ww_x\right] \dd x\notag\\
&\qquad\qquad =\int_0^1 \frac{\rho_x }{\rho} ww_x\dd x.
\end{align}
Concerning the third term, similar arguments lead to
\begin{align}
&\int_0^1 \left[  \left( \tu^2   -   \frac{\rho_x \tu}{\rho^2}\right)(\rho w)_x
+\frac{(\rho w)_{xx} \tu}{\rho} -2\rho \tu \left(\frac{w_{xx}}{\rho} -(\tu w)_x-ww_x\right) +  \frac{\rho_x }{\rho}  (\tu w)_x \right] \dd x\notag\\
&\qquad\qquad =\int_0^1 \left[  \left( \tu^2  -\left(\frac{\tu}{\rho}\right)_x -   \frac{\rho_x \tu}{\rho^2}\right)(\rho w)_x
+2 \tu_x w_{x}+2\rho \tu (\tu w)_x+2\rho \tu ww_x +  \frac{\rho_x }{\rho}  (\tu w)_x \right] \dd x\notag\\
&\qquad\qquad =\int_0^1 \left[ \tu_x w_x+ 3\rho\tu^2w_x+\tu^2\rho_xw+2\rho\tu\tu_xw+2\rho \tu w w_x+\frac{\rho_x\tu}{\rho}w_x\right]\dd x\notag\\
&\qquad\qquad =\int_0^1 \left[ \tu_x w_x+ 2\rho\tu^2w_x+2\rho \tu w w_x+\frac{\rho_x\tu}{\rho}w_x\right]\dd x.
\end{align}
Therefore, collecting all of the above we find that
\begin{align}
\mathfrak{F}(\rho,\tu, f,g)=-A^2\int_0^1 w_x\rho\,\dd x-\frac12 \int_0^1 \frac{\rho_x }{\rho} ww_x\dd x
-\frac12\int_0^1 \left[ \tu_x w_x+ 2\rho\tu^2w_x+2\rho \tu w w_x+\frac{\rho_x\tu}{\rho}w_x\right]\dd x
\end{align}
and using \eqref{eq:mean1}, \eqref{eq:poinc1}, \eqref{eq:insfdgk} and standard inequalities we deduce that
\begin{align}
&|\mathfrak{F}(\rho,\tu, f,g)|\notag\\ 
&\qquad\leq A^2\| w_x\|_{L^\infty}+\| w_x\|_{L^\infty}^3+\frac12\|(\log\rho)_x\|_{L^2}\|w_x\|^2_{L^2}+\frac12\|\tu_x\|_{L^2} \|w_{x}\|_{L^2}
+c\|w_x\|_{L^\infty}\EnD(\rho,\tu)
\end{align}
implying that
\begin{align}
\frac{\dd}{\dd t} \EnD(\rho,\tu) + \frac{1}{4} \|\tu_x\|_{L^2}^2 + \frac{A^2}{4} \| (\log \rho)_x\|_{L^2}^2
&\leq c\|w_x\|_{L^\infty}\EnD(\rho,\tu)+\| w_x\|_{L^\infty}^3\notag\\
&\quad+A^2\|w_x\|_{L^\infty}+c\|w_{x}\|_{L^2}^2\left(1+\frac{1}{A^2}\|w_{x}\|_{L^2}^2\right).
    \label{eq:E:evo:gf2}
\end{align}
Hence, \eqref{eq:energypath} follows from the standard Gronwall lemma, together with the fact
that
\begin{align}
\mathcal{E}(\rho,u)
&=\int_0^1 \left(\frac{1}{2} \rho (\tu +w)^2 + A^2\rho \log\rho \right)\dd x+\frac{1}{2}  \int_0^1 \left(   \frac{\rho_x (\tu+w)}{\rho} +   \frac{1}{2} \frac{\rho_x^2}{\rho^3} \right) \dd x\\
&\leq c \mathcal{E}(\rho,\tu)+\int_0^1 \left(\rho w^2 + \frac{\rho_x w}{\rho}  \right)\dd x
\leq  c\mathcal{E}(\rho,\tu)+c\|w_x\|^2_{L^2}.
\end{align}

\subsubsection{$H^1$ estimates on the velocity}
To prove \eqref{eq:H1path}, we multiply \eqref{eq:cNSE:3til} by $u_{xx}$ and integrate by parts, to obtain the identity
\begin{align}
\frac{\dd}{\dd t} \|\tu_x\|_{L^2}^2 +2\int_0^1 \frac{|\tu_{xx}|^2}{\rho}\dd x= 
2 \int_0^1 \tu \tu_x\tu_{xx}\dd x+2 A^2\int_0^1\frac{\rho_x}{\rho}\tu_{xx} \dd x-2\int_0^1g(\rho,\tu,w)\tu_{xx} \dd x\label{eq:H1est1}.
\end{align}
Now,
\begin{align}
2\left|\int_0^1 \tu\tu_x\tu_{xx}\dd x\right|& \leq 2 \left(\int_0^1 \frac{|\tu_{xx}|^2}{\rho}\dd x\right)^{1/2}\|\rho^{1/2} \tu\|_{L^\infty}\|\tu_x\|_{L^2}\notag\\
%& \leq 2 \left(\int_0^1 \frac{|\tu_{xx}|^2}{\rho}\dd x\right)^{1/2}\|\rho\|^{1/2}_{L^\infty}\|\tu\|_{L^\infty}\|\tu_x\|_{L^2}\notag\\
& \leq 2 \left(\int_0^1 \frac{|\tu_{xx}|^2}{\rho}\dd x\right)^{1/2}\|\rho\|^{1/2}_{L^\infty}\|\tu_x\|^2_{L^2}\notag\\
& \leq  \frac13\int_0^1 \frac{|\tu_{xx}|^2}{\rho}\dd x+c\|\rho\|_{L^\infty}\|\tu_x\|^4_{L^2},
\end{align}
and
\begin{align}
2A^2\left|\int_0^1\tu_{xx} \frac{\rho_x}{\rho} \dd x\right|
%&\leq 2A^2\left(\int_0^1 \frac{|\tu_{xx}|^2}{\rho}\dd x\right)^{1/2} \left(\int_0^1\frac{\rho_x^2}{\rho} \dd x\right)^{1/2}\notag\\
&\leq \frac13\int_0^1 \frac{|\tu_{xx}|^2}{\rho}\dd x +cA^4 \int_0^1\frac{\rho_x^2}{\rho} \dd x\notag\\
&\leq \frac13\int_0^1 \frac{|\tu_{xx}|^2}{\rho}\dd x +cA^4\|\rho\|_{L^\infty} \|(\log\rho)_x\|_{L^2}^2.
\end{align}
Concerning the last term, we find
\begin{align}
2\left|\int_0^1g(\rho,\tu,w)\tu_{xx} \dd x\right|
&=2\left|\int_0^1\rho^{1/2}\left(\frac{w_{xx}}{\rho} -(\tu w)_x-ww_x\right)\frac{\tu_{xx}}{\rho^{1/2}} \dd x\right|\notag\\
%&\leq c\int_0^1\left(\frac{w_{xx}}{\rho^{1/2}} -\rho^{1/2}\tu_x w-\rho^{1/2}\tu w_x-\rho^{1/2}ww_x\right)\dd x \left(\int_0^1 \frac{|\tu_{xx}|^2}{\rho}\dd x\right)^{1/2}\\
%&\leq c\int_0^1\left(\frac{w_{xx}}{\rho^{1/2}} -\|\rho\|_{L^\infty}^{1/2}\|\tu_x\|_{L^2} \|w\|_{L^\infty}+ \|w_x\|_{L^\infty}\|\tu_x\|_{L^2}+\|w_x\|_{L^2}\|w_x\|_{L^\infty}\right)\dd x \left(\int_0^1 \frac{|\tu_{xx}|^2}{\rho}\dd x\right)^{1/2}\\
&\leq c\left[\|\rho^{-1}\|_{L^\infty}\|w_{xx}\|^2_{L^2} +(1+\|\rho\|_{L^\infty}) \|w_x\|^2_{L^\infty}\|\tu_x\|^2_{L^2}
+\|w_x\|^4_{L^\infty}\right] \notag\\
&\quad+\frac13\int_0^1 \frac{|\tu_{xx}|^2}{\rho}\dd x.
\end{align}
Thus, \eqref{eq:H1est1} and the above estimates entail
\begin{align}
\frac{\dd}{\dd t} \|\tu_x\|_{L^2}^2 +\int_0^1 \frac{|\tu_{xx}|^2}{\rho}\dd x
&\leq c\left[\|\rho\|_{L^\infty}\|\tu_x\|^2_{L^2} +(1+\|\rho\|_{L^\infty}) \|w_x\|^2_{L^\infty}\right]\|\tu_x\|^2_{L^2}\notag\\
 &\quad +c\left[\|\rho^{-1}\|_{L^\infty}\|w_{xx}\|^2_{L^2} 
+\|w_x\|^4_{L^\infty}+A^4\|\rho\|_{L^\infty} \|(\log\rho)_x\|_{L^2}^2\right].
\end{align}
In view of \eqref{eq:energypath}, \eqref{eq:uplowpath}  and Lemma \ref{lem:enbounds}, 
we can apply the Gronwall lemma to the above inequality and deduce \eqref{eq:H1path}.

\subsection{Uniqueness of strong solutions}\label{sub:unique}
Let $(\rho,u)$ and $(r,v)$ be two solutions to \eqref{eq:cNSE:1}-\eqref{eq:cNSE:2}.
The relative entropy between $(\rho,u)$ and $(r,v)$ is defined as the functional
\begin{align}
\EnR(\rho, u | r, v) 
		&= \int_0^1 \left( \frac{1}{2}\rho (u - v)^2 + A^2\rho\log \frac{\rho}{r}\right)\dd x.
\end{align}
It is well-known that the above functional is not a distance, and it is not even symmetric. Nonetheless $\EnR(\rho, u | r, v)=0$
if and only if $(\rho,u)=(r,v)$. More precisely, the following well-known facts hold.
\begin{lemma}\label{lem:relentrbdds}
Let $(\rho,u)\in \XX$ and $(r,v)\in \XX$. Then
\begin{align}\label{eq:relentrlow}
\frac{\min\{1,A^2\}}{2}\min\left\{\|\rho^{-1}\|_{L^\infty} ,\|r^{-1}\|_{L^\infty} \right\}\left[  \| u-v\|_{L^2}^2+ \| \rho-r\|_{L^2}^2\right]
\leq \EnR(\rho, u | r, v)
\end{align}
and
\begin{align}\label{eq:relentrup}
\EnR(\rho, u | r, v) \leq \frac12 \|\rho\|_{L^\infty}  \| u-v\|_{L^2}^2 +\frac{A^2}{2}\max\left\{ \|\rho^{-1}\|_{L^\infty} , \|r^{-1}\|_{L^\infty} \right\}  \| \rho-r\|_{L^2}^2.
\end{align}
\end{lemma}

\begin{proof}
The proof is based on the fact that the function $\xi \mapsto F(\xi):=\xi \log\xi -\xi+1$ is convex.
 In particular,
\begin{align}
F'(\xi)= \log \xi, \qquad F''(\xi)= \frac{1}{\xi}, \qquad \forall \xi >0.
\end{align}
Moreover, due
to \eqref{eq:mean1}, the relative entropy can be rewritten as
\begin{align}
\EnR(\rho, u | r, v) 
		&= \int_0^1 \left( \frac{1}{2}\rho (u - v)^2 + A^2\left[F(\rho)-F(r)-(\rho-r)F'(r)\right]\right)\dd x.
\end{align}
Therefore, using Taylor's theorem, on the one hand we have the lower bound
\begin{align}
\frac{1}{2} \min\left\{\frac{1}{\xi_1},\frac{1}{\xi_2}\right\} (\xi_1-\xi_2)^2
\leq F(\xi_1)-F(\xi_2)-(\xi_1-\xi_2)F'(\xi_2)
\end{align}
while on the other hand the upper bound reads
\begin{align}
F(\xi_1)-F(\xi_2)-(\xi_1-\xi_2)F'(\xi_2)
\leq \frac{1}{2}\max\left\{\frac{1}{\xi_1},\frac{1}{\xi_2}\right\} (\xi_1-\xi_2)^2.
\end{align}
Using the above bounds in the expression for $\EnR$ yields precisely \eqref{eq:relentrlow}-\eqref{eq:relentrup}, and the proof is over.
\end{proof}
We then have the following uniqueness result.

\begin{theorem}\label{eq:unique}
Let $(\rho,u)$ and $(r,v)$ be two solutions of \eqref{eq:cNSE:1}-\eqref{eq:boundary1}, corresponding to the same initial 
condition $(\rho_0,u_0)\in \XX$.  
Then
\begin{align}
\Prb\big[(\rho(t),u(t)) =(r(t),v(t)), \ \forall t\geq 0 \big]=1.
\end{align}

\end{theorem}

\begin{proof}
We compute the time derivative of $\EnR(\rho, u | r, v) $. Firstly, notice that \eqref{eq:cNSE:2} can be rewritten as
\begin{align}
\dd u + \left( u u_x + A^2\frac{\rho_x}{\rho} \right) \dd t = \frac{u_{xx}}{\rho} \dd t + \sigma \dd W,
\end{align}
so that the difference of the velocities satisfies
\begin{align}\label{eq:cNSE:2diff}
(u-v)_t +  u u_x -vv_x+ A^2\left[\log\left(\frac{\rho}{r}\right)\right]_x = \frac{u_{xx}}{\rho}-\frac{v_{xx}}{r} ,
\end{align}
Therefore, making use of  \eqref{eq:cNSE:1} and \eqref{eq:cNSE:2diff}, we have that 
\begin{align}
\frac{\dd}{\dd t} \frac12\int_0^1 \rho (u - v)^2\dd x &= \frac12\int_0^1 \rho_t (u - v)^2\dd x +\int_0^1 \rho (u - v)(u - v)_t\dd x\notag \\
&=-\frac12\int_0^1 (\rho u)_x(u - v)^2\dd x+\int_0^1 \rho (u - v)\left(\frac{u_{xx}}{\rho}-\frac{v_{xx}}{r}\right)\dd x \notag\\
&\quad-\int_0^1 \rho (u - v)(u u_x -vv_x)\dd x-A^2\int_0^1 \rho (u - v) \left[\log\left(\frac{\rho}{r}\right)\right]_x\dd x\notag\\
&=\int_0^1 \rho u(u - v)(u-v)_x\dd x- \| (u - v)_x\|_{L^2}^2
+\int_0^1 \frac{\rho-r}{r} (v -u)  v_{xx} \dd x\notag\\
&\quad-\int_0^1 \rho (u - v)(u u_x -vv_x)\dd x-A^2\int_0^1 \rho (u - v) \left[\log\left(\frac{\rho}{r}\right)\right]_x\dd x\notag\\
&=-\| (u - v)_x\|_{L^2}^2+\int_0^1 \frac{\rho-r}{r} (v -u)  v_{xx} \dd x- \int_0^1 v_x  \rho   (u - v)^2 \dd x\notag\\
&\quad-A^2\int_0^1 \rho (u - v) \left[\log\left(\frac{\rho}{r}\right)\right]_x\dd x.\label{eq:cont1}
\end{align}
On the other hand, due to \eqref{eq:cNSE:1} once more, we compute
\begin{align}
\frac{\dd}{\dd t} \int_0^1 \rho\log\left(\frac{\rho}{r}\right)\dd x&=\int_0^1 \rho_t\log\left(\frac{\rho}{r}\right)\dd x
+\int_0^1 \left(\rho_t- \frac{\rho}{r}r_t\right)\dd x\notag\\
&=-\int_0^1 (\rho u)_x\log\left(\frac{\rho}{r}\right)\dd x+\int_0^1\frac{\rho}{r}(r v)_x\dd x\notag\\
&=\int_0^1 \rho u\left[\log\left(\frac{\rho}{r}\right)\right]_x\dd x+\int_0^1\rho v_x\dd x+\int_0^1\frac{\rho}{r}r_x v\dd x\notag\\
&=\int_0^1 \rho (u - v) \left[\log\left(\frac{\rho}{r}\right)\right]_x.\label{eq:cont2}
%&=\int_0^1 \left([\rho P'(\rho)+P(\rho)]\rho_t -[r P'(r)+P(r)]r_t-(\rho-r)_t\left[ rP'(r)+P(r)\right]-(\rho-r)\frac{p'(r)}{r}r_t\right)\dd x\notag\\
%&=\int_0^1 \left([\rho P'(\rho)+P(\rho) -rP'(r)-P(r) ]\rho_t -(\rho-r)\frac{p'(r)}{r}r_t\right)\dd x\notag\\
%&=-\int_0^1 \left([\rho P'(\rho)+P(\rho) -rP'(r)-P(r) ](\rho u)_x -(\rho-r)\frac{p'(r)}{r}(rv)_x\right)\dd x\notag\\
%&=\int_0^1 \left(\left(\frac{p'(\rho)}{\rho}\rho_x -\frac{p'(r)}{r}r_x \right)\rho u +(\rho-r)\frac{p'(r)}{r}(rv)_x\right)\dd x\notag\\
%&=\int_0^1 \left(\left(\frac{p(\rho)_x}{\rho} -\frac{p(r)_x}{r} \right)\rho u +(\rho-r)\frac{p'(r)}{r}(rv)_x\right)\dd x\notag\\
%&=\int_0^1 \rho (u - v)\left(\frac{p(\rho)_x}{\rho} -\frac{p(r)_x}{r}  \right)\dd x+\int_0^1 \left(p(\rho)_x v 
%+ p'(r)\rho v_x-p'(r)rv_x-p(r)_xv\right)\dd x\notag\\
%&=\int_0^1 \rho (u - v)\left(\frac{p(\rho)_x}{\rho} -\frac{p(r)_x}{r}  \right)\dd x-\int_0^1 v_x \left(p(\rho)-p(r)-p'(r)(\rho-r)\right)\dd x.\label{eq:cont2}
\end{align}
Therefore, combining \eqref{eq:cont1} and \eqref{eq:cont2}, we arrive at 
\begin{align}\label{eq:relentrequal}
\frac{\dd}{\dd t}\EnR(\rho, u | r, v) +\| (u - v)_x\|_{L^2}^2&=\int_0^1 \frac{\rho-r}{r} (v -u)  v_{xx} \dd x- \int_0^1 v_x  \rho   (u - v)^2 \dd x.\end{align}
We now estimate each term in the right-hand side above. We preliminary notice that for any smooth function 
$g$ such that $g(x_0)=0$ for some $x_0\in [0,1]$, there holds the inequality
\begin{align}\label{eq:poinc2}
\| g\|_{L^\infty} \leq \left( \int_0^1 \frac{|g_x|^2}{\rho}\dd x\right)^{1/2},
\end{align}
where $\rho$ can be replaced by any positive function with mass at most 1 (and, in particular, by $r$). Indeed,
\begin{align}
|g(x)|=\left|\int_{x_0}^x g_y\dd y\right|=\left|\int_{x_0}^x \rho^{1/2}\frac{g_y}{\rho^{1/2}}\dd y\right|\leq \left(\int_{x_0}^x \rho\dd y\right)^{1/2}
\left(\int_{x_0}^x\frac{|g_y|^2}{\rho}\dd y\right)^{1/2}\leq \left( \int_0^1 \frac{|g_x|^2}{\rho}\dd x\right)^{1/2}.
\end{align}
Concerning the first term in the right-hand side of \eqref{eq:relentrequal}, using standard inequalities and \eqref{eq:relentrlow} 
%applied to $v_x$ (which necessarily vanishes due to the Dirichlet boundary conditions)
we obtain
\begin{align}
\int_0^1 \frac{\rho-r}{r} (v -u)  v_{xx} \dd x
&\leq \|r^{-1}\|^{1/2}_{L^\infty}\|\rho-r\|_{L^2}\|v-u\|_{L^\infty}\left(\int_0^1\frac{|v_{xx}|^2}{r}\dd x\right)^{1/2} \notag\\
&\leq\|r^{-1}\|_{L^\infty}\left(\int_0^1\frac{|v_{xx}|^2}{r}\dd x\right)\|\rho-r\|_{L^2}^2 +\frac12\|(u-v)_x\|_{L^2}^2\notag\\
&\leq \frac{2}{\min\{1,A^2\}}\frac{\|r^{-1}\|_{L^\infty} }{\min\left\{\|\rho^{-1}\|_{L^\infty} ,\|r^{-1}\|_{L^\infty} \right\}}\left(\int_0^1\frac{|v_{xx}|^2}{r}\dd x\right)\EnR(\rho, u | r, v)\notag \\
&\quad+\frac12\|(u-v)_x\|_{L^2}^2\notag\\ 
&\leq \frac{2}{\min\{1,A^2\}}\max\left\{1, \frac{\|r^{-1}\|_{L^\infty} }{\|\rho^{-1}\|_{L^\infty} }\right\}\left(\int_0^1\frac{|v_{xx}|^2}{r}\dd x\right)\EnR(\rho, u | r, v)\notag \\
&\quad+\frac12\|(u-v)_x\|_{L^2}^2,
\end{align}
where the last line follows from \eqref{eq:relentrlow}. For the second term, we use \eqref{eq:poinc1} and \eqref{eq:poinc2} to deduce that
\begin{align}
\int_0^1 v_x  \rho   (u - v)^2 \dd x
&\leq \|v_x\|_{L^\infty}\left(\int_0^1  \rho   (u - v)^2 \dd x\right)^{1/2}\|(u-v)_x\|_{L^2}\notag\\
&\leq \left(\int_0^1\frac{|v_{xx}|^2}{r}\dd x\right)\EnR(\rho, u | r, v)+\frac12\|(u-v)_x\|_{L^2}^2.
\end{align}
Hence, from \eqref{eq:relentrequal} it follows that
\begin{align}\label{eq:relentrequal2}
\frac{\dd}{\dd t}\EnR(\rho, u | r, v) \leq \left(1+\frac{2}{\min\{1,A^2\}}\max\left\{1, \frac{\|r^{-1}\|_{L^\infty} }{\|\rho^{-1}\|_{L^\infty} }\right\}\right)\left(\int_0^1\frac{|v_{xx}|^2}{r}\dd x\right)\EnR(\rho, u | r, v).
\end{align}
Thanks to Proposition \ref{prop:allestimates2}, the term multiplying $\EnR(\rho, u | r, v)$ in  the right-hand side above is locally integrable
in time. 
We now apply the standard Gronwall lemma to \eqref{eq:relentrequal2} and use the
fact that 
$$
\EnR(\rho, u | r, v)|_{t=0}=0
$$ 
by assumption to conclude that the two solutions are indistinguishable.

%Now, for each $K>0$, define the sequence of stopping times
%\begin{align}\label{eq:stoptime}
%\tau_K=\inf_{t\geq 0}\left\{ \int_0^t\left(1+\frac{2}{\min\{1,A^2\}}\max\left\{1, \frac{\|r^{-1}\|_{L^\infty} }{\|\rho^{-1}\|_{L^\infty} }\right\}\right)\left(\int_0^1\frac{|v_{xx}|^2}{r}\dd x\right)\dd s>K\right\}.
%\end{align}
%Thanks to \eqref{eq:intgrab}, $\tau_K\to \infty$ almost surely as $K\to \infty$.
%For each fixed $K>0$, 
%\begin{align}\label{eq:relentrequal3}
%\E\sup_{s\in [0,t\wedge \tau_K]}\EnR(\rho, u | r, v) \leq \E\EnR(\rho_0, u_0 | r_0, v_0)  +\E\int_0^t\left(1+\frac{2}{\min\{1,A^2\}}\max\left\{1, \frac{\|r^{-1}\|_{L^\infty} }{\|\rho^{-1}\|_{L^\infty} }\right\}\right)\left(\int_0^1\frac{|v_{xx}|^2}{r}\dd x\right)\EnR(\rho, u | r, v)\dd s
%\end{align}
%\alert{This should do it, but need to double check. Estimates might not be optimal}
\end{proof}

\subsection{Continuous dependence of solutions}\label{sub:conti}

Regarding the continuous dependence of solutions, the picture is more complicated, and can be proven in the following situation.
\begin{theorem}\label{eq:contdiprel}
Let $\{(\rho^n,u^n)\}_{n\in \NN}$ be a sequence of solutions to \eqref{eq:cNSE:1}-\eqref{eq:boundary1} with initial data
$\{(\rho_0^n,u_0^n)\}_{n\in \NN}$ such that
\begin{align}\label{eq:notsure}
\sup_{n\in \NN}\EnD(\rho_0^n,u_0^n)<\infty,
\end{align}
and let  $(\rho,u)$ be a a solution to \eqref{eq:cNSE:1}-\eqref{eq:boundary1} with initial datum
$(\rho_0,u_0)\in \XX$. If 
\begin{align}\label{eq:initialconv}
(\rho_0^n,u_0^n)\to (\rho_0,u_0) \quad \text{in } L^2\times L^2,
\end{align}
then
\begin{align}
(\rho^n,u^n)\to (\rho,u) \quad \text{a.s. in } L^2\times L^2,
\end{align}
for all $t\geq 0$.
\end{theorem}

\begin{proof}
We use the approach in Theorem \ref{eq:unique}. Arguing in the same way as we did to derive \eqref{eq:relentrequal2}, we arrive at
\begin{align}\label{eq:relentrequaln}
\frac{\dd}{\dd t}\EnR(\rho^n, u^n | \rho, u) \leq \left(1+\frac{2}{\min\{1,A^2\}}\max\left\{1, \frac{\|\rho^{-1}\|_{L^\infty} }{\|(\rho^n)^{-1}\|_{L^\infty} }\right\}\right)\left(\int_0^1\frac{|u_{xx}|^2}{\rho}\dd x\right)\EnR(\rho^n, u^n | \rho, u).\notag\\
\end{align}
Due to the lack of symmetry of the relative entropy, it is important that we consider the above quantity, and not the other possible
$\EnR(\rho, u | \rho^n, u^n)$.
Let 
$$
R=\max\left\{\sup_{n\in \NN}\EnD(\rho_0^n,u_0^n), \EnD(\rho_0,u_0)\right\}<\infty.
$$ 
In view of Proposition \ref{prop:allestimates2}, we have  $n$-independent almost sure
bounds of the form
\begin{align}
\sup_{t\in[0,T]}\int_0^t\left(\int_0^1\frac{|u_{xx}|^2}{\rho}\dd x\right)\dd s\leq C\left(T,\sigma W, R, \|(u_0)_x\|_{L^2}\right),
\end{align}
and
\begin{align}
\sup_{t\in[0,T]}\frac{\|\rho^{-1}\|^2_{L^\infty} }{\|\rho_n^{-1}\|^2_{L^\infty}}\leq  C\left(T,\sigma W,R\right).
\end{align}
The point is that the latter bound does not depend on $ \|(u^n_0)_x\|_{L^2}$,
which is not, in general, under control uniformly in $n\in \NN$. Thanks to Lemma \ref{lem:relentrbdds}, the fact that 
$\EnD(\rho^n_0, u_0^n)\leq R$ and \eqref{eq:initialconv}, we have that 
$$
\lim_{n\to\infty} \EnR(\rho_0^n, u_0^n | \rho_0, u_0)=0.
$$
Thus, the standard Gronwall lemma applied to \eqref{eq:relentrequaln} implies
$$
\lim_{n\to \infty}\EnR(\rho^n(t), u^n(t) | \rho(t), u(t))=0, \qquad \text{a.s.}
$$
for all $t\geq 0$. Hence, a further application of Lemma \ref{lem:relentrbdds} leads to the conclusion of the proof. 
\end{proof}

\section{Invariant measures}\label{sec:invar}
In this section, we prove of the main results of this article, namely the existence of invariant measures 
to \eqref{eq:cNSE:1}-\eqref{eq:cNSE:2}. As mentioned earlier, we will work in the phase space
\begin{align}
\XX=\left\{(\rho,u)\in H^1(0,1)\times H^1_0(0,1): \int_0^1 \rho(x) \dd x = 1, \ \rho>0\right\}.
\end{align}
However, in view of Theorem \ref{eq:contdiprel}, the correct topology does not seem to be the natural one induced
by the product norm in $\XX$, since Lemma \ref{lem:relentrbdds} suggests that continuous dependence for our system holds
in a weaker sense. We make this precise here below.

%
%In the following, we will denote by $\XXL$ the set $\XX$ endowed with $L^2 \times L^2$ metric. Notice that $\XXL$ is clearly
%a metric space, but it not complete, due to both the open condition $\rho >0$ and the fact that the Sobolev space $H^1$ is not
%closed in $L^2$. By $\mathcal{B}(\XXL)$, we denote the family of Borel subsets of $\XXL$. 
% With the symbol $M_b(\XXL)$ (resp. 
%$C_c(\XXL)$) we refer to the set of all real valued bounded measurable (resp. continuous and compactly supported)
%functions on $\XXL$. Finally, $\mathfrak{P}(\XXL)$ is the set of all probability measures on $\XXL$.

\subsection{The Markovian framework}\label{sub:Markov}
We denote by $\rho(t;\rho_0), u(t;u_0)$ the unique solution to \eqref{eq:cNSE:1}-\eqref{eq:cNSE:2}, with (deterministic)
initial data $(\rho_0, u_0)\in \XX$. For a set $B\in \mathcal{B}(\XX)$, we define the transition functions
\begin{align}
\Ma_t(\rho_0, u_0, B)=\Prb ((\rho(t;\rho_0), u(t;u_0)) \in B),
\end{align}
for any $t\geq 0$.
For $t\geq 0$, define the Markov semigroup
\begin{equation}
\Ma_t \phi(\rho_0, u_0)= \E\phi(\rho(t;\rho_0), u(t;u_0)), \qquad \phi\in M_b(\XXL).
\end{equation}
The usual semigroup properties
$$
\Ma_0=\text{identity on }M_b(\XXL)
$$
and
$$
\Ma_{t+s}=\Ma_t\circ\Ma_s, \qquad \forall t,s\geq 0,
$$
follow from the existence and uniqueness results for our system. Now, it is clear from the estimates in Propositions \ref{prop:allestimates}
and \ref{prop:allestimates2} that 
$$
\Ma_t:M_b(\XXL)\to M_b(\XXL),
$$
namely, $\Ma_t$ is well defined on measurable bounded functions. The Feller property, namely that $\Ma_t$ maps $C_b(\XXL)$
into $C_b(\XXL)$ is more delicate, and we are not able to prove it at the moment.  In particular, it does not follow from Theorem \ref{eq:contdiprel}, due to the additional requirement \eqref{eq:notsure}. However, Theorem \ref{eq:contdiprel} suggests 
the following definition.

\begin{definition}
Let $\phi \in M_b(\XXL)$. We say that $\phi:\XX\to \RR$ belongs to the class $\GG$ if
$$
\lim_{n\to\infty}\phi(\rho^n,u^n)=\phi(\rho,u),
$$
whenever 
\begin{align}\label{eq:convseq}
\{(\rho^n,u^n)\}_{n\in \NN}\subset \XXL, \quad \lim_{n\to \infty}(\rho^n,u^n)=(\rho,u) \quad \text{in } \XXL,  \quad \text{and}\quad
\sup_{n\in \NN}\EnD(\rho^n,u^n)<\infty.
\end{align}
\end{definition}
From the above definition and Lemma \ref{lem:enbounds}, it is clear that 
$$
C_b(\XXL)\subset \GG\subset C_b(\XXH).
$$
It turns out that $\GG$ is also invariant under $\Ma_t$. 

\begin{lemma}\label{lem:GtoG}
The class $\GG$ is invariant under the Markov semigroup $\Ma_t$, namely
$$
\Ma_t:\GG\to \GG.
$$
\end{lemma}

\begin{proof}
Let $\{(\rho_0^n,u_0^n)\}_{n\in \NN}\subset \XXL$ be a sequence, complying with \eqref{eq:convseq}, namely
\begin{align}\label{eq:rew8}
\lim_{n\to \infty}(\rho_0^n,u_0^n)=(\rho_0,u_0) \quad \text{in } \XXL,  \quad \text{and}\quad
M:=\sup_{n\in \NN}\EnD(\rho_0^n,u_0^n)<\infty.
\end{align}
In light of Theorem \ref{eq:contdiprel},
\begin{align}\label{eq:dgsafhtr2}
\lim_{n\to \infty}(\rho(t;\rho_0^n),u(t;u_0^n))=(\rho(t;\rho_0),u(t;u_0)) \quad \text{a.s. in } \XXL,\quad \forall t\geq 0.
\end{align}
Moreover, in view of \eqref{eq:entro2} and \eqref{eq:rew8}, there exists a constant $\bar{M}:=\bar{M}(t,M,\sigma)$ independent of $n\in \NN$
such that 
$$
\E\EnD(\rho(t;\rho^n_0), u(t;u^n_0))\leq \bar{M},
$$
which, by Chebyshev's inequality, it implies that 
\begin{align}\label{eq:cheb}
\Prb \left[\EnD(\rho(t;\rho^n_0), u(t;u^n_0)) > R\right]\leq \frac{\bar{M}}{R},
\end{align}
for any $R>0$. Fix $\phi \in \GG$, and let 
\begin{align}\label{eq:phibound}
M_\phi:=\sup_{(\rho,u)\in\XXL} |\phi(\rho,u)|<\infty.
\end{align}
To show that $\Ma_t \phi\in \GG$, we are required to prove that for any fixed $t>0$ there holds
$$
\lim_{n\to \infty}\Ma_t \phi(\rho^n_0, u^n_0)=
\Ma_t \phi(\rho_0, u_0).
$$
Writing $\EnD_n=\EnD(\rho(t;\rho^n_0), u(t;u^n_0))$ for short, we have by definition that
\begin{align}
\Ma_t \phi(\rho^n_0, u^n_0)-\Ma_t \phi(\rho_0, u_0)&=\E\left[\phi(\rho(t;\rho^n_0), u(t;u^n_0))-\phi(\rho(t;\rho_0), u(t;u_0))\right]\notag\\
&=\E\left[(\phi(\rho(t;\rho^n_0), u(t;u^n_0))-\phi(\rho(t;\rho_0), u(t;u_0)))\indFn{\EnD_n\leq R}\right]\notag\\
&\quad +\E\left[(\phi(\rho(t;\rho^n_0), u(t;u^n_0))-\phi(\rho(t;\rho_0), u(t;u_0)))\indFn{\EnD_n> R}\right].
\end{align} 
In light of \eqref{eq:cheb}-\eqref{eq:phibound},
\begin{align}\label{eq:tgerwgs}
|\E\left[(\phi(\rho(t;\rho^n_0), u(t;u^n_0))-\phi(\rho(t;\rho_0), u(t;u_0)))\indFn{\EnD_n> R}\right]|\leq \frac{2 \bar{M}M_\phi}{R}.
\end{align}
Since $\phi\in \GG$, we also deduce by \eqref{eq:rew8} and \eqref{eq:dgsafhtr2} that
$$
\lim_{n\to \infty} (\phi(\rho(t;\rho^n_0), u(t;u^n_0))-\phi(\rho(t;\rho_0), u(t;u_0)))\indFn{\EnD_n\leq R}=0, \qquad \text{a.s.},
$$
and therefore by the bounded convergence theorem that
\begin{align}\label{eq:gfgsfgsgsad}
\lim_{n\to\infty}\E\left[(\phi(\rho(t;\rho^n_0), u(t;u^n_0))-\phi(\rho(t;\rho_0), u(t;u_0)))\indFn{\EnD_n\leq R}\right]=0.
\end{align}
Let us now arbitrarily fix $\eps>0$, and pick $R_\eps=4\bar{M}M_\phi/\eps $. Invoking \eqref{eq:tgerwgs}, we deduce that
$$
\E\left[(\phi(\rho(t;\rho^n_0), u(t;u^n_0))-\phi(\rho(t;\rho_0), u(t;u_0)))\indFn{\EnD_n> R}\right]\leq \frac\eps2.
$$
Moreover, from \eqref{eq:gfgsfgsgsad} there exists $n_\eps\in \NN$ such that
$$
|\E\left[(\phi(\rho(t;\rho^n_0), u(t;u^n_0))-\phi(\rho(t;\rho_0), u(t;u_0)))\indFn{\EnD_n\leq R}\right]|<\frac\eps2, \qquad \forall n\geq n_\eps.
$$
Thus,
$$
|\Ma_t \phi(\rho^n_0, u^n_0)-\Ma_t \phi(\rho_0, u_0)|< \eps, \qquad \forall t\geq0, n\geq n_\eps,
$$
and the proof is over.
\end{proof}

%\alert{Check this}
%\alert{Theorem \ref{eq:contdiprel}, together with Lemma \ref{lem:relentrbdds} shows that $\Ma_t$ maps
% $C_b(\XX)$ into itself, namely $\Ma_t$ is Feller on $\XX$. Also, $\Ma_t^*$ will denote the dual semigroup acting on measures}
%

%Associated to \eqref{eq:cNSE:1}-\eqref{eq:cNSE:2} is the so-called \emph{Markov semigroup}, acting on the 

%$\{\MA_t\}_{t\geq 0}$,
%defined on the space $M_b(L^2)$  as 
%\begin{equation}
%\MA_t \phi(f_0)= \EE\phi(f^\nu(t,f_0)), \qquad \phi\in M_b(L^2),\ t\geq 0.
%\end{equation}
%Here, we stress the dependence on the initial datum by writing $f^\nu(t,f_0)$ for the solution
%to \eqref{eq:SPDEviscous} emanating from $f_0$. Since $f^\nu(t,f_0)$ depends continuously on $f_0$,
%it follows that $\{\MA_t\}_{t\geq 0}$ is \emph{Feller}, namely, it also maps  $C_b(L^2)$ to itself. 

\subsection{Tightness of time-averaged measures}\label{sub:tight}
We prove the existence of an invariant measure for $\Ma_t$ via the classical Krylov-Bogoliubov procedure.
However, since
we are working in a non-complete metric space,  and with a Markov semigroup that is not Feller, some details
do not follow directly from the well-known theory. In what follows, the parameter $A>0$ appearing in \eqref{eq:cNSE:2} 
is fixed, but we suppress 
the dependence on it of all the quantities (except for bounds), in order to keep the notation as simple as possible. 
As initial conditions to
our problem, we fix $\rho_0=1$ and $u_0=0$. Notice that the corresponding energy vanishes, namely $\EnD(1, 0)=0$.
For $T>0$, define the time-averaged measure on $\XX$ by
\begin{align}\label{eq:timeave}
\mu_T(B)=\frac1T\int_0^T\Prb(\rho(t;1),u(t;0)\in B) \dd t,
\end{align}
where $B\in \mathcal{B}(\XXL)$. The first step is to prove tightness of the family $\{\mu_T\}_{T>0}$. We will make use of the following lemma.

\begin{lemma}\label{lem:sup}
Assume that $\rho\in H^1$ is such that $\rho> 0$ and
\begin{align}
\int_0^1\rho\, \dd x=1, \qquad M^2:=\int_0^1 \frac{\rho_x^2}{\rho^2}\dd x=\int_0^1 [(\log\rho)_x]^2\dd x <\infty.
\end{align}
Then, for $x\in [0,1]$, we have the pointwise bounds
\begin{align}\label{eq:Ubdd}
\e^{-M}\leq\rho(x)\leq  \e^M,
\end{align}
and, moreover
\begin{align}
\int_0^1 \rho_x^2 \dd x\leq  M^2\e^{2M}.
\end{align}
\end{lemma}

\begin{proof}
 Firstly, we notice 
since $\rho$ is continuous and has mean 1, there exists $x_0\in (0,1)$ such
that $\rho(x_0)=1$. Therefore, for every $x\in [0,1]$ we have
\begin{align}
|\log \rho(x)|=\left|\int_{x_0}^x (\log\rho)_y \dd y \right|\leq \left[\int_0^1[(\log\rho)_y]^2\dd y\right]^{1/2}=M, \qquad \forall x\in[0,1].
\end{align}
As a consequence,
\begin{align}
\e^{-M}\leq \rho(x)\leq \e^{M}, \qquad \forall x\in [0,1],
\end{align}
proving  \eqref{eq:Ubdd}. Moreover,
\begin{align}
\int_0^1 \rho_x^2\dd x\leq \|\rho\|_{L^\infty}^2\int_0^1\frac{\rho_x^2}{\rho^{2}}\dd x.
\end{align}
The claim then follows from \eqref{eq:Ubdd}.
\end{proof}

As a consequence, tightness of $\{\mu_T\}_{T>0}$ follows in a straightforward manner.

\begin{proposition}\label{prop:tightseq}
The family of probability measures $\{\mu_T\}_{T>0}\subset \mathfrak{P}(\XXL)$ is tight. Hence,
there exists a subsequence $T_j\to \infty$ and a measure $\mu\in \mathfrak{P}(\XXL)$ such that
\begin{align}\label{eq:measprop1}
\lim_{j\to \infty} \int_\XX \phi\, \dd\mu_{T_j} = \int_\XX \phi\, \dd\mu, \qquad \forall \phi \in C_b(\XXL). 
\end{align}
\end{proposition}

\begin{proof}
For any fixed $R,S\geq 1$, define the sets
\begin{align}
K_R=\left\{(\rho,u)\in \XX: \int_0^1 u_x^2 \dd x+ \int_0^1 \frac{\rho_x^2}{\rho^{2}}\dd x \leq R^2\right\}
\end{align}
and 
\begin{align}\label{eq:setcs}
C_S=\left\{(\rho,u)\in \XX:   \int_0^1 u_x^2 \dd x + \int_0^1\rho_x^2\dd x+ \|\rho\|_{L^\infty}+\|\rho^{-1}\|_{L^\infty}\leq S\right\}.
\end{align}
Note that since $C_S$ is bounded in $H^1\times H^1_0$ and is closed, it is compact in $\XXL$.
By Lemma \ref{lem:sup}, we have that if $(\rho,u)\in K_R$, then
\begin{align}
\int_0^1\rho_x^2\dd x\leq R^2\e^{2R},\quad \int_0^1 u_x^2 \dd x \leq R^2,\quad \|\rho\|_{L^\infty}\leq \e^{R},\quad \|\rho^{-1}\|_{L^\infty}\leq \e^R,
\end{align}
In particular, if  $R\geq 1$, we obtain
\begin{align}\label{eq:srdef}
\int_0^1\rho_x^2\dd x+\int_0^1 u_x^2 \dd x + \|\rho\|_{L^\infty}+\|\rho^{-1}\|_{L^\infty}\leq S_R:=4R^2\e^{2R},
\end{align}
which therefore translates into the set inclusion
\begin{align}\label{eq:inclusion}
K_R\subset C_{S_R}.
\end{align}
In view of \eqref{eq:inclusion} and Chebyshev's inequality we have
\begin{align}
\mu_T(C_{S_R})&\geq \mu_T (K_R)=1-\mu_T (\XX\setminus K_R)\\
&=1-\frac{1}{T}\int_0^T \Prb \left[\int_0^1 \frac{\rho_x^2}{\rho^{2}}\dd x+ \int_0^1 u_x^2 \dd x > R^2\right]\\
&\geq 1-\frac{1}{R^2} \frac{1}{T} \int_0^T \E \left[\int_0^1 \frac{\rho_x^2}{\rho^{2}}\dd x+ \int_0^1 u_x^2 \dd x\right].
\end{align}
In light of the energy inequality \eqref{eq:entro2} and the fact that $\EnD(1, 0)=0$, it follows that
\begin{align}\label{eq:meastcr1}
\mu_T(C_{S_R})\geq 1-\frac{\|\sigma\|_{L^\infty}^2}{\min\{1,A^2\}}\frac{1}{R^2}.
\end{align}
Hence, the family $\{\mu_T\}_{T>0}$ is tight, and therefore there exists a subsequential limit $\mu\in \mathfrak{P}(\XXL)$. Note that this
uses the direction of Prokhorov's theorem that does not require completeness (see \cite{Billingsley99}*{Theorem 6.1}).
Now, since $C_{S_R}$ is a closed set, the Portmanteau theorem implies
\begin{align}\label{eq:meastcr2}
\mu(\XX)\geq \mu(C_{S_R}) \geq\limsup_{T\to\infty}\mu_T(C_{S_R}) \geq 1-\frac{\|\sigma\|_{L^\infty}^2}{\min\{1,A^2\}}\frac{1}{R^2}.
\end{align}
Notice again that this does not require the metric space $\XXL$ to be complete (see \cite{Billingsley99}*{Theorem 2.1}).
Thus
\begin{align}
\mu(\XX)\geq \lim_{R\to \infty}\mu(C_{S_R})=1.
\end{align}
The proof is therefore concluded.
\end{proof}

\subsection{Invariance of the limit measure}
We aim to prove the following proposition.
\begin{proposition}\label{prop:inv}
For every fixed $A>0$, the Markov semigroup $\{\Ma_t\}_{t\geq 0}$ associated to \eqref{eq:cNSE:1}-\eqref{eq:boundary1}
possesses an invariant probability measure $\mu_A\in \mathfrak{P}(\XXL)$. Furthermore,
\begin{align}\label{eq:measprop2}
\int_{\XX} \left[A^2\|(\log\rho)_x\|_{L^2}^2+\|u_x\|_{L^2}^2\right]\mu_A(\rho,u) \leq \|\sigma\|^2_{L^\infty}.
\end{align}
\end{proposition}

The main issue here is that $\Ma_t$ is not known to be Feller, and therefore we cannot directly apply \eqref{eq:measprop1}
to $\Ma_t\phi$. As a first step, we extend property \eqref{eq:measprop1} to the class $\GG$.

\begin{lemma}\label{lem:exttoG}
Let $\{\mu_{T_j}\}_{j\in\NN}\subset \mathfrak{P}(\XXL)$ be the convergent sequence in Proposition \ref{prop:tightseq}. Then
\begin{align}
\lim_{j\to \infty} \int_\XX \phi\, \dd\mu_{T_j} = \int_\XX \phi\, \dd\mu, \qquad \forall \phi \in \GG. 
\end{align}
\end{lemma}

\begin{proof}
Let $\phi\in\GG$, arbitrarily fix $R\geq 1$, and consider the compact sets $C_{S_R}$ from \eqref{eq:setcs}, where
$S_R$ is given by \eqref{eq:srdef}. Since $\phi \in \GG$, $\phi$ is bounded on $\XXL$, so that
$$
M_\phi:=\sup_{(\rho,u)\in\XXL} |\phi(\rho,u)|<\infty.
$$
Moreover,
the restriction of $\phi$ to $C_{S_R}$ is continuous (in 
the $L^2\times L^2$ metric) on $C_{S_R}$, for every $R\geq 1$. Since  $C_{S_R}$ is a closed set, Tietze extension theorem (see e.g. 
\cite{Kelley55}) guarantees the existence of a function $\tilde\phi_R\in C_b(\XXL)$ such that
\begin{align}\label{eq:extprop}
\tilde\phi_R=\phi \quad \text{on } C_{S_R}, \qquad \sup_{(\rho,u)\in\XXL}|\tilde\phi_R(\rho,u)|\leq M_\phi, \qquad \forall R\geq 1.
\end{align}
Now,
\begin{align}
\left|\int_{\XX}\phi\, \dd\mu_{T_j}-\int_\XX \phi\, \dd\mu\right|\leq 
\left|\int_{\XX}(\phi-\tilde\phi_R)\, \dd\mu_{T_j}\right|+
\left|\int_{\XX}\tilde\phi_R\, \dd\mu_{T_j}-\int_\XX \tilde\phi_R\, \dd\mu\right|+
\left|\int_{\XX}(\tilde\phi_R-\phi)\, \dd\mu\right|
\end{align}
Fix $\eps >0$.
By \eqref{eq:meastcr1} and the above \eqref{eq:extprop}, there exists $R_\eps\geq 1$ such that
$$
\left|\int_{\XX}(\phi-\tilde\phi_{R_\eps})\, \dd\mu_{T_j}\right|=\left|\int_{\XX\setminus C_{S_{R_\eps}}}(\phi-\tilde\phi_{R_\eps})\, \dd\mu_{T_j}\right|\leq 2M_\phi \mu_{T_j}(\XX\setminus C_{S_{R_\eps}})\leq \frac{2M_\phi\|\sigma\|_{L^\infty}^2}{\min\{1,A^2\}}\frac{1}{R_\eps^2}<\frac{\eps}{3}.
$$
uniformly for $j\in \NN$, and, similarly, in light of \eqref{eq:meastcr2},
$$
\left|\int_{\XX}(\tilde\phi_{R_\eps}-\phi)\, \dd\mu\right|<\frac{\eps}{3}.
$$
Since $\tilde\phi_{R_\eps}\in C_b(\XXL)$, property \eqref{eq:measprop1} implies the existence of $j_\eps\in \NN$ such that
$$
\left|\int_{\XX}\tilde\phi_{R_\eps}\, \dd\mu_{T_j}-\int_\XX \tilde\phi_{R_\eps}\, \dd\mu\right| <\frac{\eps}{3},\qquad \forall j\geq j_\eps.
$$
Hence, for every $\eps>0$, there exists $j_\eps\in \NN$ such that
$$
\left|\int_{\XX}\phi\, \dd\mu_{T_j}-\int_\XX \phi\, \dd\mu\right|<\eps, \qquad \forall j\geq j_\eps,
$$
which proves the claim.
\end{proof}

We are now in the position to prove invariance of the limiting measure and the bound \eqref{eq:measprop2}.
\begin{proof}[Proof of Proposition \ref{prop:inv}]
As before, $A>0$ is fixed and we suppress 
the various dependences on it. Let $\phi \in C_b(\XXL)$ be fixed. Thanks to Lemma \ref{lem:GtoG},
$$
\Ma_t\phi \in \GG, \qquad \forall t>0.
$$
In view of Lemma \ref{lem:exttoG}, we then have that
$$
\lim_{j\to \infty} \int_\XX \Ma_t\phi\, \dd\mu_{T_j} = \int_\XX \Ma_t\phi\, \dd\mu, \qquad \forall t\geq 0.
$$
Hence, by \eqref{eq:timeave}, we conclude that
\begin{align}
 \int_\XX \Ma_t\phi\, \dd\mu
&=\lim_{j\to \infty} \int_\XX \Ma_t\phi\, \dd\mu_{T_j} 
=\lim_{j\to \infty} \frac{1}{T_j}\int_0^{T_j} \Ma_{t+s} \phi(1,0)\dd s
=\lim_{j\to \infty} \frac{1}{T_j}\int_t^{T_j+t} \Ma_{s} \phi(1,0)\dd s\notag\\
&=\lim_{j\to \infty} \left(  \frac{1}{T_j}\int_0^{T_j} \Ma_{s} \phi(1,0)\dd s+  \frac{1}{T_j}\int_{T_j}^{T_j+t} \Ma_{s} \phi(1,0)\dd s
-\frac{1}{T_j}\int_0^{t} \Ma_{s} \phi(1,0)\dd s\right)\notag\\
&=\lim_{j\to \infty} \int_\XX\phi\, \dd\mu_{T_j} = \int_\XX \phi\, \dd\mu\notag,
\end{align}
so that invariance follows from the arbitrariness of $\phi$. Lastly, \eqref{eq:measprop2} is deduced
directly from \eqref{eq:entro2}. Indeed, is $(\rho^S,u^S)\in \XX$ is a statistically stationary solution
distributed as $\mu$, the we can use \eqref{eq:entro2} to derive the bound
\begin{align}
\E\|u^S_x\|^2_{L^2} + A^2\E\|(\log \rho^S)_x\|_{L^2}^2
\leq \| \sigma \|^2_{L^\infty},
\end{align}
which is precisely \eqref{eq:measprop2} after the usual change of variables. The proof is over.
\end{proof}

\subsection*{Acknowledgements}
The authors would like to thank Jacob Bedrossian, Sandra Cerrai, Dave Levermore, Jonathan Mattingly, Sam Punshon-Smith and 
Mohammed Ziane for helpful discussions. We also thank the program ``New Challenges
in PDE: Deterministic Dynamics and Randomness in High Infinite Dimensional Systems'', held at the 
Mathematical Science Research Institute (MSRI) in the Fall 2015, where some initial steps in the work were formulated.
MCZ was in part supported by the National Science Foundation under the award DMS-1713886.  K.T.  gratefully acknowledges the support  in part by the National Science Foundation under the award DMS-1614964.
NGH was in part supported by the National Science Foundation under the award DMS-1313272 and Simons Foundation 515990.

\end{document}